\documentclass[10pt, a4paper]{amsart}

\usepackage{graphicx} 
\usepackage{amsmath}
\usepackage{amssymb}
\usepackage{tikz-cd}
\usepackage{amsthm}
\usepackage{bm}
\usepackage{hyperref}
\usepackage{tikz}
\usepackage{caption}
\usepackage{adjustbox}
\usepackage[dvipsnames]{xcolor}
\usepackage{mathtools}
\usepackage{stmaryrd}

\newtheorem{theorem}{Theorem}[section]
\newtheorem{proposition}[theorem]{Proposition}

\newtheorem{corollary}[theorem]{Corollary}

\theoremstyle{definition}
\newtheorem{example}[theorem]{Example}
\newtheorem{definition}[theorem]{Definition}

\theoremstyle{remark}
\newtheorem{remark}{Remark}[section]

\renewcommand{\d}{\mathrm{d}}
\newcommand{\A}{\mathcal{A}}
\newcommand{\tL}{\tilde{\mathcal{L}}}
\renewcommand{\O}{\mathfrak{o}}

\newcommand{\La}{\mathcal{L}}
\newcommand{\ZA}{\zeta^{\operatorname{ask}}_{\bm{\mu}_\A/\O}(s)}
\newcommand{\ZAK}{\zeta^{\operatorname{ask}}_{\bm{\mu}_\A/\mathcal{O}_K}(s)}

\title[Ask zeta functions of central hyperplane arrangements]{Ask zeta functions of central hyperplane arrangements}
\author{Alec Schmutz}
\date{\today}

\thanks{The author was supported by the Deutsche Forschungsgemeinschaft (DFG, German Research Foundation) – Project-ID~491392403 – TRR~358}

\begin{document}

\begin{abstract}
    Given a central hyperplane arrangement $\A$ defined over a field of characteristic zero, we construct matrices of linear forms whose local ask zeta functions are recovered by the Igusa local zeta function of the cone over $\mathcal{A}$. Our construction extends a previously established connection between ask zeta function of hypergraphs and the Igusa local zeta function of Boolean arrangements. From a combinatorial standpoint, we introduce the truncated flag Hilbert--Poincar\'e series of $\A$, obtained as a rank specialisation of the flag Hilbert--Poincar\'e series of the cone over $\mathcal{A}$. Whenever $\A$ admits good reduction over the finite field $\mathbb{F}_q$, suitable substitutions of the variables of its truncated flag Hilbert--Poincar\'e series recover the local ask zeta functions associated with $\A$. Such formulae provide a means to study the analytic properties of the ask zeta functions considered, as well as to derive their reduced and topological relatives.
\end{abstract}

\maketitle

\section{Introduction}
The main objective of this paper is to develop the relationship between two classes of rational generating functions: first, ask zeta functions associated with matrices of linear forms (cf. \cite{RossmannAsk/18}); second, Igusa local zeta functions associated with central hyperplane arrangements (cf. \cite{MaglioneVoll/24}). This relation provides a broad class of ask zeta function that can be studied by means of a combinatorially defined multivariate rational function: the flag Hilbert--Poincar\'e series.

\subsection{Background}
We provide a brief overview of the rational generating functions considered before turning our attention to the main results of this paper.
\subsubsection{Ask zeta functions} Introduced in \cite{RossmannAsk/18}, ask zeta functions are rational generating functions that enumerate average sizes of kernels of matrices of linear forms defined over finite rings. Specifically, fixing a ring $R$ and integers $d,e\geq 1$, we let $$A=A(\bm{X})\in M_{d\times e}(R[X_1, \dots , X_m])$$ be a matrix whose entries consist of \emph{linear forms}, where $\bm{X}=(X_i)_{i=1}^m$ is a set of algebraically independent variables over $R$. Considering a finite $R$-algebra $S$, the {\bf{a}}verage {\bf{s}}ize {\bf{k}}ernel of $A(\bm{X})$ is defined by $$\mathrm{ask}_S(A(\bm{X}))= \frac{1}{\lvert S\rvert^m}\sum_{x\in S^m} \lvert \ker(A(x))\rvert.$$ 
Setting $R=\O $ to be a compact discrete valuation ring (cDVR) with maximal ideal $\mathfrak{p}$, the (algebraic) ask zeta function of $A(\bm{X})$ is given by the following Dirichlet series: 
\begin{equation}\label{eq:AlgebraicAsk}
Z^{\mathrm{ask}}_{A/\O }(T)= \sum_{k\geq 0} \mathrm{ask}_{\O /\mathfrak{p}^k}(A(\bm{X})) \ T^{k} \in \mathbb{Q}\llbracket T\rrbracket.
\end{equation}

Substituting $T=q^{-s}$ in equation~(\ref{eq:AlgebraicAsk}) for a complex variable $s$, where $q$ is the size of the finite residue field of $\O $, yields the analytic ask zeta function of $A(\bm{X})$: $$\zeta^{\mathrm{ask}}_{A/\O }(s):=Z^{\operatorname{ask}}_{A/\O }(q^{-s}).$$
Such zeta functions are of group-theoretic interest, as they relate to the (conjugacy) class-counting and orbit-counting zeta functions of unipotent group schemes; cf.~\cite[Thm.~1.7]{RossmannAsk/18}.

\subsubsection{Flag Hilbert--Poincar\'e series} 
\label{Sec:FlagHilbert}
A hyperplane arrangement $\A$ over a field $\mathbb{K}$ is a finite collection of affine hyperplanes in $\mathbb{K}^{d}$. We denote $\dim(\A):= d$. The flag Hilbert--Poincar\'e series of $\mathcal{A}$ is a multivariate rational function introduced in \cite{MaglioneVoll/24}. To define it, let $\La(\A)$ denote the poset of all possible non-empty intersections of the hyperplanes of $\A$, ordered by reverse inclusion. The bottom element $\hat{0}$ is the ambient vector space of the arrangement. The \emph{Poincar\'e polynomial} associated with $\A$ is defined as
$$\pi_{\A}(t)= \sum_{x\in \mathcal{L}(\mathcal{A})} \mu(x) (-t)^{\mathrm{rk}(x)},$$
where $\mu$ denotes the Möbius function on $\La(\A)$, and $\mathrm{rk}$ denotes the rank function on $\La(\A)$; cf.~\cite[Def.~2.48]{orlik_arrangements_1992}. For any $x\in \La(\A)$, the subarrangement $\A_x$ and the restriction $\A^x$ are given by $$\A_x=\{H\in \A : H\supseteq x\} \quad\quad \text{and} \quad\quad \A^x=\{H\cap x: H\in \A\backslash \A_x, \ H\cap x\neq \varnothing\},$$ where we set $\A_{\varnothing}:=\A$. In particular, $\A_{\hat{0}}$ is the empty arrangement. Given $x,y\in \La(\A)$, we can interlace these operations to get $\A_x^y:=(\A_x)^y=(\A^y)_x$.
For a poset $P$, let $\Delta(P)$ denote the abstract simplicial complex whose simplices are given by flags in $P$; cf.~\cite[Chap.~8.4]{petersen_eulerian_2015}. Setting $P=\mathcal{L}(\A)$, the flag Poincar\'e polynomial associated with any $F=(x_1<\dots <x_l)\in \Delta(\La(\A))$ is defined as a product of Poincar\'e polynomials:
\begin{equation*}
\pi_F(X)=\prod_{i=0}^{l}\pi_{\A_{x_{i+1}}^{x_{i}}}(X),
\end{equation*}
where we set $x_0=\hat{0}$ and $x_{l+1}=\varnothing$.
The flag Hilbert--Poincar\'e series of $\mathcal{A}$ is the given by the following multivariate rational function: 
\begin{equation*}
\mathsf{fHP}_{\mathcal{A}}(X, (T_x)_{x\in \mathcal{L}(\A)\backslash\{\hat{0}\}})=\sum_{F\in \Delta(\mathcal{L}(\A)\backslash \{\hat{0}\})} \pi_F(X)\prod_{x\in F}\frac{T_x}{1-T_x}\in \mathbb{Q}(\bm{T})[X].
\end{equation*}
This rational function is of interest due to its connection with Igusa local zeta functions of products of affine linear polynomials, which we now explain.

Let $\A$ be a hyperplane arrangement defined over a field of characteristic zero. In view of \cite[Prop.~6.8.11]{oxley_matroid_2011}, $\A$ is representable over a number field $K$. We let $\mathcal{O}_K$ denote the ring of integers of $K$. Fixing a basis $\{e_1, \dots, e_{\dim(\A)}\}$ of $K^{\dim(\A)}$, we may express any $H\in \A$ as the vanishing locus of a linear affine polynomial $L_H(X_1, \dots, X_{\dim(\A)})\in \mathcal{O}_K[\bm{X}]$. This way, we can identify $\A$ with the collection of polynomials $\{L_H(\bm{X}): H\in \A\}$. Note that these conventions are in accordance with \cite[Sec.~1.1]{MaglioneVoll/24}. In the sequel, any cDVR $\O$ will be assumed to posses an $\mathcal{O}_K$-algebra structure. Examples include $\O$ being a finite extension of the completion of $\mathcal{O}_K$ at a non-zero prime ideal $\mathfrak{p}\subset \mathcal{O}_K$ (characteristic zero), or the power series $\mathbb{F}_q\llbracket t\rrbracket$ for positive characteristics.

The \emph{analytic zeta function of $\mathcal{A}$} is defined by the ensuing Igusa local zeta function:
\begin{equation*}
\zeta_{\mathcal{A}(\O )}(\bm{s})=\int_{\O^{\dim(\A)}} \prod_{x\in \mathcal{L}(\A)\backslash \{\hat{0}\}} \lVert \mathcal{A}_x\rVert^{s_x}\ \d\mu,
\end{equation*}
where $s_x$ is a complex variable associated with every $x\in \mathcal{L}(\mathcal{A})\backslash \{\hat{0}\}$, the measure $d\mu$ is the normalised Haar measure on $\O^{\dim(\A)}$, and $\lVert \mathcal{A}_x\rVert$ denotes the maximal norm on the set of affine forms corresponding to the hyperplanes containing $x$. Under suitable substitutions of the variables, the flag Hilbert--Poincar\'e series and the analytic zeta function of $\A$ determine one another; see~\cite[Thm.~B]{MaglioneVoll/24}. We note that the flag Hilbert--Poincar\'e series, together with its various specialisations, have recently gained traction as a combinatorial object of independent interest; see for instance \cite{KuhneMaglione/23, DorpalenMaglioneStump/25, Stump/25}.

\subsection{Main results}
As in Section~\ref{Sec:FlagHilbert}, we let $\A$ denote a hyperplane arrangement defined over a number field $K$ and $\O$ be a cDVR and an $\mathcal{O}_K$-algebra. We assume moreover that $\A$ is \emph{central}, meaning that $\cap_{H\in \A} H\neq \varnothing$. 
\subsubsection{Ask zeta functions of central hyperplane arrangements}
\label{Sec:BeyondBoolean}
A relationship between ask zeta functions of hypergraphs and Igusa zeta functions of hyperplane arrangements was previously known in \cite[Sec.~4.8]{MaglioneVoll/24}, though it was only established for Boolean arrangements. Here, we extend this relationship much further and explicate a relationship between ask zeta functions and Igusa zeta functions over all central hyperplane arrangements, which subsumes the initial connection observed by Maglione and Voll.
To be specific, we construct matrices of linear forms associated with $\A$, such that their ask zeta functions are recovered by the analytic zeta function of $c\A$, the cone over $\A$.

Let $$\bm{Z}=\{Z_{(H,x)} : (H, x)\in \A\times (\La(\A)\backslash\{\hat{0}\}) \text{ such that } H\supseteq x\}$$ denote a set of algebraically independent variables over $\O$. We construct a matrix of linear forms $A_{\A}^{\bm{1}}(\bm{Z})\in M_{\dim(\A)\times \lvert \La(\A)\rvert}(\O[\bm{Z}])$ as follows: with regards to our choice of basis $\{e_1, \dots, e_{\dim(\A)}\}$ for $\O^{\dim(\A)}$, for any $1\leq i\leq \dim(\A)$ and any $ x\in \La(\A)$, we set
$$ (A_{\A}^{\bm{1}}(\bm{Z}))_{(i,x)}:=\sum_{H\in \A_x} L_H(e_i)Z_{(H, x)}.$$
More generally, given a tuple $\bm{\mu}:=(\mu_x)_{x\in \La(\A)}$ of nonnegative integers, we may repeat the column indexed by $x\in \La(\A)$ in the above matrix $\mu_x$ times, introducing new variables for every additional column; cf. Definition~\ref{DefIntWeighIncRep} and Remark~\ref{Rem:MatrixLinFormsA}. We refer to $\bm{\mu}$ as the vector of \emph{intersection multiplicities}, and denote the resulting matrix of linear forms by $A_{\A}^{\bm{\mu}}(\bm{Z})$. Note that such matrices are reminiscent of the matrices of linear forms associated with hypergraphs considered in \cite[Sec.~3.2]{RossmannVoll/24}.
To simplify notation, the ask zeta function associated with $A_{\A}^{\bm{\mu}}(\bm{Z})$ is denoted by  $$\zeta^{\operatorname{ask}}_{\bm{\mu}_\A/\O}(s):=\zeta_{A_{\A}^{\bm{\mu}}/\O}^{\operatorname{ask}}(s).$$

\begin{example}
\label{ExampleBraidMatrix}
    The braid arrangement $\mathrm{A}_3$ is defined by the set $$\left\{(X_i-X_j): 1\leq i<j\leq 4\right\}.$$ The intersection lattice $\La(\operatorname{A}_3)$ is isomorphic to the poset of partitions on four elements, which we denote by  $\Pi_{4}$; cf. \cite[Prop.~2.9]{orlik_arrangements_1992}. Let $\bm{\mu}=(\mu_x)_{x\in \Pi_4}$ be the vector of intersection multiplicities given by
    \begin{equation*}
    \mu_x=\begin{cases}
        1 &\text{if $x\in \{123, 124, 12|34, 134, 234\}$}, \\
        0 & \text{otherwise.}
    \end{cases}
    \end{equation*}
    The Hasse diagram of $\Pi_4$ is provided in Figure \ref{Fig:IntLatticeBraid}, where the non-dashed edges correspond to incidence relations between the elements of $$\{123, 124, 12|34, 134, 234\}$$ and rank-one elements.
    The matrix of linear forms $A_{\operatorname{A}_3}^{\bm{\mu}}(\bm{Z})$ is then given by
    $$\begin{pmatrix}
         Z_{12}^{123}+Z_{13}^{123} & Z_{12}^{124}+Z_{14}^{124} & Z_{12}^{12|34} & Z_{13}^{134}+Z_{14}^{134} & 0\\
         Z_{23}^{123}-Z_{12}^{123} & Z_{24}^{124}-Z_{12}^{124} & -Z_{12}^{12|34} & 0 & Z_{23}^{234}+Z_{24}^{234}\\
         -Z_{13}^{123}-Z_{23}^{123} & 0 & Z_{34}^{12|34} & Z_{34}^{134}-Z^{134}_{13} & Z_{34}^{234}-Z_{23}^{234} \\ 
         0 & -Z_{14}^{124}-Z_{24}^{124} & -Z_{34}^{12|34} & -Z^{134}_{14}-Z^{134}_{34} & -Z_{24}^{234}-Z_{34}^{234}
    \end{pmatrix},$$
    where the rows are indexed by $\{e_1,e_2,e_3,e_4\}$, and the columns are indexed by the partitions $\{123,124,12|34, 134, 234\}$. In view of the chosen intersection multiplicities, there are $14$ variables appearing in $A_{\operatorname{A}_3}^{\bm{\mu}}(\bm{Z})$, each associated with a non-dashed edge in Figure~\ref{Fig:IntLatticeBraid}.
    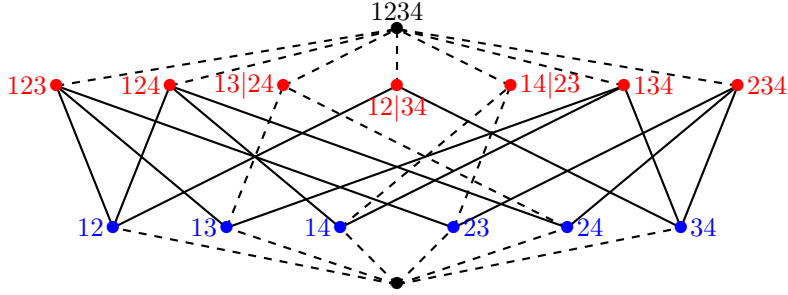
\begin{figure}
    \centering
    \begin{tikzpicture}[scale=0.75]
    \draw[thick, dashed] (0,0) -- (1,1); 
        \draw[thick, dashed] (0,0) -- (3,1);  
        \draw[thick, dashed] (0,0) -- (-3,1); 
        \draw[thick, dashed] (0,0) -- (-5,1); 
        \draw[thick, dashed] (0,0) -- (-1,1); 
        \draw[thick, dashed] (0,0) -- (5,1); 

        \draw[thick, dashed] (0,3.5) -- (0,4.5);  
        \draw[thick, dashed] (2,3.5) -- (0,4.5); 
        \draw[thick, dashed] (4,3.5) -- (0,4.5); 
        \draw[thick, dashed] (6,3.5) -- (0,4.5); 
        \draw[thick, dashed] (-2,3.5) -- (0,4.5); 
        \draw[thick, dashed] (-4,3.5) -- (0,4.5);
        \draw[thick, dashed] (-6,3.5) -- (0,4.5);
        
        \draw[thick] (-5,1) -- (-6,3.5); 
        \draw[thick] (-5,1) -- (-4,3.5); 
        \draw[thick] (-5,1) -- (0,3.5); 
        \draw[thick] (-3,1) -- (-6,3.5); 
        \draw[thick, dashed] (-3,1) -- (-2,3.5); 
        \draw[thick] (-3,1) -- (4,3.5); 
        \draw[thick] (-1,1) -- (-4,3.5); 
        \draw[thick, dashed] (-1,1) -- (2,3.5);
        \draw[thick] (-1,1) -- (4,3.5);
        \draw[thick] (1,1) -- (-6,3.5); 
        \draw[thick, dashed] (1,1) -- (2,3.5);
        \draw[thick] (1,1) -- (6,3.5);
        \draw[thick] (3,1) -- (-4,3.5); 
        \draw[thick, dashed] (3,1) -- (-2,3.5);
        \draw[thick] (3,1) -- (6,3.5);
        \draw[thick] (5,1) -- (0,3.5); 
        \draw[thick] (5,1) -- (4,3.5);
        \draw[thick] (5,1) -- (6,3.5);
        
        \node at (0,0) {\Large{\textbullet}};
        
        \node at (1,1) [blue] {\Large{\textbullet}};
        \node at (1,1) [blue, right] {$23$};
        \node at (-3,1) [blue] {\Large{\textbullet}};
        \node at (-3,1) [blue, left] {$13$};
        \node at (3,1) [blue] {\Large{\textbullet}};
        \node at (3,1) [blue, right] {$24$};
        \node at (-1,1) [blue] {\Large{\textbullet}};
        \node at (-1,1) [blue, left] {$14$};
        \node at (-5,1) [blue] {\Large{\textbullet}};
        \node at (-5,1) [blue, left] {$12$};
        \node at (5,1) [blue] {\Large{\textbullet}};
        \node at (5,1) [blue, right] {$34$};
        
        \node at (0,3.5) [red] {\Large{\textbullet}};
        \node at (0,3.5) [red, below] {$12|34$};
        \node at (2,3.5) [red] {\Large{\textbullet}};
        \node at (2,3.5) [red, right] {$14|23$};
        \node at (-2,3.5) [red] {\Large{\textbullet}};
        \node at (-2,3.5) [red, left] {$13|24$};
        \node at (-4,3.5) [red] {\Large{\textbullet}};
        \node at (-4,3.5) [red, left] {$124$};
        \node at (4,3.5) [red] {\Large{\textbullet}};
        \node at (4,3.5) [red, right] {$134$};
        \node at (-6,3.5) [red] {\Large{\textbullet}};
        \node at (-6,3.5) [red, left] {$123$};
        \node at (6,3.5) [red] {\Large{\textbullet}};
        \node at (6,3.5) [red, right] {$234$};
        
        \node at (0,4.5) {\Large{\textbullet}};
        \node at (0,4.5) [above] {$1234$};
        
    \end{tikzpicture}
    \captionof{figure}{The intersection lattice $\mathcal{L}(\mathrm{A}_{3})\simeq \Pi_4$.}
    \label{Fig:IntLatticeBraid}
\end{figure}
\end{example}

The cone over $\mathcal{A}$, denoted by $c\A$, is a central hyperplane arrangement of dimension $\dim(\mathcal{A})+1$. It is defined by the linear polynomials $$\left\{\prescript{h}{}{L_H}(X_1, \dots, X_{\dim(\A)}): H\in \A\right\}\cup \{X_0\},$$
where $\prescript{h}{}{L_H}(\bm{X})$ denotes the homogenisation of $L_H(\bm{X})$. The hyperplane $H_0=\ker(X_0)$ is referred to as the \emph{additional hyperplane}. Since $\A$ is assumed to be central, we have $\prescript{h}{}{L_H}(\bm{X})=L_H(\bm{X})$, and an isomorphism of posets 
\begin{equation*}
\La((c\A)^{H_0})\simeq \La(\A)
\end{equation*}
induced by (\ref{eq:LConeRestr}). Whenever $\underline{x}\in \La((c\A)^{H_0})$, we denote by $x\in \La(\A)$ the corresponding linear subspace under this isomorphism.
Given $n\in \mathbb{Z}$, we define multivariate polynomials for each $\underline{x}\in \La(c\A)$:
    \begin{equation}
    \label{FunctionCentralA}
    f_{n, \underline{x}}(X,Y, (Z_y)_{y\in \La(\A)}):=\begin{cases}
        X+Y-Z_{\hat{0}}-n-1 & \text{if $\underline{x}=H_0,$} \\
        -Z_{x} &\text{if $\underline{x}> H_0$,} \\ 
        0 &\text{otherwise}.
    \end{cases}
    \end{equation}
Our first main result is the following:
\begin{theorem}
\label{ThmMainAskIgusa}
    Let $\A$ be a central hyperplane arrangement defined over a number field $K$. Let $\bm{\mu}=(\mu_x)_{x\in \La(\mathcal{A})}$ be a tuple of nonnegative integers, and set $m=\sum_{x\in \La(\A)}\mu_x$. If $\O$ is a cDVR and an $\mathcal{O}_K$-algebra with residue field of size $q$, then
    \begin{equation*}
    \ZA=(1-q^{-1})^{-1}\zeta_{c\A(\O)}\left( (f_{\dim(\A), \underline{x}}(s, m, \bm{\mu}))_{\underline{x}\in \La(c\A)\backslash\{\hat{0}\}}\right).
    \end{equation*}
\end{theorem}
In the event that $\A=\mathrm{A}_1^n$ is the Boolean arrangement of rank $n$, we have $c\mathrm{A}_1^n=\mathrm{A}^{n+1}_1$. Hence, applying Theorem~\ref{ThmMainAskIgusa} recovers \cite[Prop.~4.10]{MaglioneVoll/24}; cf. Remark~\ref{RemarkA1andHypergraph}.

\subsubsection{Truncated flag Hilbert--Poincar\'e series}
Let $\mathfrak{p}$ be the maximal ideal of the cDVR $\O$, and let $\mathbb{F}_q=\O/\mathfrak{p}$ denote the finite residue field. The hyperplane arrangement $\A$ is said to have good reduction over $\mathbb{F}_q$ if $\La(\A)\simeq \La(\A(\O/\mathfrak{p}))$, where $\A(\O/\mathfrak{p})$ denotes the reduction of $\A$ modulo $\mathfrak{p}$. In view of \cite[Prop.~5.13]{Stanley/07}, $\A$ has good reduction over $\mathbb{F}_q$ for almost all prime characteristics. Under this good reduction condition, combining Theorem~\ref{ThmMainAskIgusa} with \cite[Thm.~B]{MaglioneVoll/24} yields an expression for $\zeta_{\bm{\mu}_\A/\O}^{\operatorname{ask}}(s)$ in terms of suitable, albeit unwieldy, substitutions for the variables of $\mathsf{fHP}_{c\A}(X, \bm{T})$. We shall see, however, that the same zeta function can be recovered from much simpler substitutions of a ``truncated analogue" of the flag Hilbert--Poincar\'e series of $\A$, which we now introduce.

For any flag $F\in \Delta(\La(\A))$, let $\min(F)$ denote the minimal element of $F$, and set $\min(\varnothing)=\varnothing$. In particular, whenever $\min(F)=\hat{0}$, we have $\pi_{\A_{\min(F)}}(X)=1$. We define the \emph{truncated flag Poincar\'e polynomial} to be $$\underline{\pi}_F(X):=\frac{\pi_F(X)}{\pi_{\A_{\min(F)}}(X)} \in \mathbb{Z}[X].$$
\begin{definition}
\label{Def:TruncfHP}
Let 
$\bm{T}=(T_x)_{x\in \La(\A)}$ be indeterminates. 
The truncated flag Hilbert--Poincar\'e series associated with $\A$ is 
$$\mathsf{mfHP}_{\mathcal{A}}(X, (T_x)_{x\in \mathcal{L}(\A)})=\sum_{F\in \Delta(\mathcal{L}(\mathcal{A}))}\underline{\pi}_F(X) \prod_{x\in F}\frac{T_x}{1-T_x}\in \mathbb{Q}(\bm{T})[X].$$
\end{definition}
Whenever $\A$ is central, appending the top element $\hat{1}$ to any flag shows that 
\begin{equation}
\label{eq:mfHPcentral}
\mathsf{mfHP}_{\A}(X, \bm{T})=\frac{1}{1-T_{\hat{1}}}\sum_{F\in \Delta(\La(\A)\backslash\{\hat{1}\})} \underline{\pi}_F(X)\prod_{x\in F}\frac{T_x}{1-T_x}.
\end{equation}
Note that our truncated version differs from $\mathsf{fHP}_{\A}(X, \bm{T})$ in two aspects: first, we incorporate the bottom element of $\La(\A)$ into the order complex; second, we truncate each of the flag polynomials $\pi_F(X)$ by dividing out the first factor in its product. Also note that the specialisation $$\mathsf{mfHP}_{\A}(0, \bm{T})=\operatorname{Hilb}(\mathbb{F}[\Delta(\La(\A))]; \bm{T})$$ is the fine Hilbert series of the Stanley-Reisner ring of $\Delta(\La(\A))$; cf.~\cite[Chap.~10.6]{petersen_eulerian_2015}. Our second main result shows that the truncated flag Hilbert--Poincar\'e series of any central arrangement $\A$, defined now over a field of characteristic zero, is obtained as a specific rank specialisation of the flag Hilbert--Poincar\'e series associated with the cone over $\A$. Recall that in the central case, we have an isomorphism of posets $\La(\A)\simeq \La((c\A)^{H_0})$ induced by (\ref{eq:LConeRestr}).
\begin{theorem}
\label{Thm:mfHPandfHP}
    Let $\mathcal{A}$ be a central hyperplane arrangement defined over a field of characteristic zero. Then,
    \begin{equation}
    \label{Eq:substitutionsconefHP}
    (1+X)\mathsf{mfHP}_{\mathcal{A}}(X, (T_x)_{x\in \mathcal{L}(\A)})=\mathsf{fHP}_{c\A}\left(X, ((-X)^{\operatorname{rk}(y)})_{y\not\geq H_0}, (T_x)_{x\geq H_0}\right).
    \end{equation}
\end{theorem}
Utilising Theorem~\ref{Thm:mfHPandfHP}, we can derive a functional equation under inversion of the variables (``self-reciprocity result") for the truncated flag Hilbert--Poincar\'e series of central arrangements defined over fields of characteristic zero. To be precise, the following result is obtained by combining Theorem~\ref{Thm:mfHPandfHP} with the self-reciprocity result \cite[Thm.~A]{MaglioneVoll/24} applied to $c\A$.
\begin{corollary}
\label{Cor:mfHPSelf}
    Let $\mathcal{A}$ be a central hyperplane arrangement defined over a field of characteristic zero. Then 
    $$\mathsf{mfHP}_{\A}(X^{-1}, (T_x^{-1})_{x\in \mathcal{L}(\A)})=-(-X)^{-\mathrm{rk}(\A)}T_{\hat{1}}\mathsf{mfHP}_{\A}(X, \bm{T}).$$
\end{corollary}
Corollary~\ref{Cor:mfHPSelf} is leveraged in Section~\ref{Sec:FunEquations} as a means to derive local functional equations for the ask zeta functions associated with $\A$ under inversion of the prime power $q$. The result we obtain is in accordance with \cite[Thm.~4.18]{RossmannAsk/18}, which establishes the local functional equations under inversion of $q$ for \emph{all} ask zeta functions.

Under good reduction assumptions, combining Theorem~\ref{ThmMainAskIgusa}, \cite[Thm.~B]{MaglioneVoll/24}, and Theorem~\ref{Thm:mfHPandfHP} yields explicit, combinatorial formulae for the ask zeta functions associated with the matrices of linear forms constructed from central hyperplane arrangements.
\begin{corollary}
\label{CorollarymfHP}
    Let $\A$ be a central hyperplane arrangement defined over a number field $K$, and let $\bm{\mu}=(\mu_x)_{x\in \La(\A)}$ be a tuple of nonnegative integers. If $\O$ is a cDVR and an $\mathcal{O}_K$-algebra with residue field of size $q$ such that $\A$ has good reduction over $\mathbb{F}_q$, then  $$\ZA=\mathsf{mfHP}_{\A}(-q^{-1}, (q^{-s}q^{\dim(\A)-\operatorname{rk}(y)-\sum_{x\not\leq y} \mu_x})_{y\in \mathcal{L}(\A)}).$$
\end{corollary}
We may view Corollary~\ref{CorollarymfHP} as a generalisation of \cite[Thm.~C]{RossmannVoll/24}: setting $\A=\mathrm{A}_1^n$ yields the combinatorial formula in terms of weak orders associated with the ask zeta functions of hypergraphs; see Example~\ref{Ex:HypergraphmfHP}. In view of the Cograph Modelling Theorem \cite[Thm~D]{RossmannVoll/24}, it is still of interest to determine whether the truncated flag Hilbert--Poincar\'e series of other hyperplane arrangements can recover the class counting zeta functions of unipotent group schemes associated with non-cographs.

In the same manner as \cite[Thm.~F]{RossmannVoll/24}, Corollary~\ref{CorollarymfHP} has remarkable consequences regarding the analytic properties of the ask zeta functions considered. For a central hyperplane arrangement defined over a number field $K$, we consider the Dirichlet series 
\begin{equation*}
\ZAK=\sum_{I\subseteq \mathcal{O}_K} \operatorname{ask}_{\mathcal{O}_K/I}\left(A_{\A}^{ \bm{\mu}}(\bm{Z})\right)\lvert \mathcal{O}_K/I\rvert^{-s},
\end{equation*}
where $I$ ranges over the non-zero ideals of the ring of integers $\mathcal{O}_K$ of $K$. We denote the abscissa of convergence of this Dirichlet series by $\alpha(\A, \bm{\mu})$, which determines the degree of polynomial growth of the summatory function of its coefficients; cf.~\cite[Chap.~8]{apostol_modular_1990}. Rossmann \cite[Thm.~4.20]{RossmannAsk/18} proved that $\alpha(\A, \bm{\mu})$ is a possibly negative rational number not larger than $\dim(\A)+1$. Our next result shows that $\alpha(\A, \bm{\mu})$ is in fact an integer and is determined explicitly by the combinatorics of $\A$. 
\begin{theorem}
\label{Thm:GlobAnalProp}
    Let $\A$ be a central hyperplane arrangement defined over a number field $K$. For any tuple of nonnegative integers $\bm{\mu}=(\mu_x)_{x\in \La(\A)}$, the abscissa of convergence of $\ZAK$ is a positive integer. It satisfies $$\alpha(\A, \bm{\mu})=\max_{x\in \La(\A)}\left\{\dim(\A)+1-\mathrm{rk}(x)-\sum_{y\not\leq x}\mu_y\right\}\in [\dim(\A)-\mathrm{rk}(\A)+1, \dim(\A)+1],$$ and is independent of $K$.
\end{theorem}

\subsection{Notation}
We let $\mathbb{N}:=\{1,2,\dots\}$ be the set of positive integers. For any $n\in \mathbb{N}$, we set $[n]:=\{1, \dots, n\}$. Given a poset $P$, we write $\hat{0}$ for the bottom element, and $\hat{1}$ for the top element, provided they exist. We let $\Delta(P)$ denote the order complex of $P$, and $2^P$ denote the power set of $P$. Given a flag $F=(x_1<\dots < x_l)\in \Delta(P)$, we let $\min(F):=x_1$ and $\max(F):=x_l$. The concatenation of two flags $F, F^{\prime}\in \Delta(P)$ such that $\max(F)\leq \min(F^{\prime})$ is denoted by $F\Vert F^{\prime}$. Given a hyperplane arrangement $\A$, $\La(\A)$ is the intersection lattice of $\A$, and we set $\tL(\A):=\La(\A)\backslash \{\hat{0}\}$. The dimension of $\A$ is denoted by $d(\A):=\dim(\A)$, and the rank of any element $x\in \La(\A)$ is $r_\A(x):=\mathrm{rk}_\A(x)$. Whenever confusion is unlikely, we omit the index $\A$ when referring to the rank function. 

We use $K$ to denote a number field, and $\mathcal{O}_K$ for its ring of integers. Throughout, $\O$ denotes a compact discrete valuation ring (cDVR) of arbitrary characteristic admitting an $\mathcal{O}_K$-algebra structure. Its maximal ideal is denoted by $\mathfrak{p}$, and we denote its residue field by $\mathbb{F}_q=\O/\mathfrak{p}$. 

\section{Preliminaries}
\label{Sec:Preliminaries}

\subsection{Ask zeta functions}
We provide a brief summary of the essential theory surrounding ask zeta functions, all of which can be found in more detail in \cite{RossmannAsk/18, RossmannAskII/20, RossmannVoll/24}. Let $R$ be a ring, and consider a matrix of linear forms $$A:= A(\bm{Z})\in M_{d\times e}(R[Z_1, \dots, Z_m]).$$ The matrix $A$ gives rise to an $R$-module homomorphism $$\theta_A\colon  \begin{array}{rcl}
   R^m & \to  & \operatorname{Hom}_R(R^d, R^e) \\
    u & \mapsto & (v\mapsto v^{\intercal}A(u))
\end{array},$$ which we refer to as a \emph{module representation}. Given a finite set $U$, we let $RU$ denote the free $R$-module of rank $\lvert U\rvert$. Letting $U, V$ and $W$ be finite sets, any module homomorphism $$\theta\colon R U\rightarrow \operatorname{Hom}_R(RV, RW)$$
is encoded by a matrix of linear forms in the following manner: 
    \begin{equation}
    \label{eq:MatrixLinForms}
    A_{\theta}:=\left(\sum_{u\in U}Z_u u\right) (\theta\otimes{R[(Z_u)_{u\in U}]})\in \text{Hom}_{R[\bm{Z}]}(R[\bm{Z}]V, R[\bm{Z}]W),
    \end{equation}
    where $\bm{Z}=(Z_u)_{u\in U}$ denotes a set of algebraically independent variables over $R$.
In particular, whenever $R=\O$ is a cDVR, we set $$Z_{\theta/\O}^{\mathrm{ask}}(T):= Z^{\mathrm{ask}}_{A_{\theta}/\O}(T);$$ cf. equation~(\ref{eq:AlgebraicAsk}). The \emph{Knuth dual} of $\theta$ is given by
$$
\theta^{\circ}\colon \begin{array}{rcl} RV & \to & \text{Hom}(RU, RW) \\ v & \mapsto & (u \mapsto v(u\theta))
\end{array},$$ and the matrix of linear forms associated with $\theta^{\circ}$ is denoted by 
\begin{equation}
\label{KnuthDual}
C_{\theta}(\bm{X}):=A_{\theta^{\circ}}(\bm{X}),
\end{equation}
where $\bm{X}=(X_v)_{v\in V}$ is a set of algebraically independent variables over $R$. Setting $R=\O $ for a cDVR $\O $ with residue field of size $q$, \cite[Cor.~2.10]{RossmannVoll/24} expresses the corresponding ask zeta function in terms of a $p$-adic integral:
    \begin{equation}
    \label{Eq:AskCoker}
    \zeta_{\theta/\O}^{\text{ask}}(s)=(1-q^{-1})^{-1}\int_{\O  V\times \O }\lvert y\rvert^{s-1+\lvert W\rvert -\lvert V\rvert} \lvert \mathrm{Coker}(C_{\theta}(\bm{X}))\otimes_{\O [\bm{X}]}(\O /y)_{\bm{x}}\rvert\, \d\mu(\bm{x},y), 
    \end{equation}
    where $(\O/y)_{\bm{x}}$ denotes the $\O[\bm{X}]$-module $\O/y$, whose scalar multiplication is defined via the composition $\O[\bm{X}]\rightarrow \O\rightarrow \O/y$, where the first map is the evaluation map at $\bm{x}$, and the second map is the quotient map. Considering sizes of cokernels in the above integrand to express sizes of kernels is motivated by the right-exactness of tensor products; cf.~\cite[Rem.~2.9]{RossmannVoll/24}.

\subsection{Hyperplane arrangements}
We provide some additional prerequisites regarding hyperplane arrangements. For a detailed account of the theory of hyperplane arrangements, consult~\cite{orlik_arrangements_1992, Stanley/07}. 

Let $\A=\{H_1, \dots, H_m\}$ be an arrangement of hyperplanes defined over some field $K$.
The intersection poset of $\A$ is given by 
$$\mathcal{L}(\A)=\left\{\bigcap_{i\in I} H_i\neq \varnothing : \varnothing\neq I\subseteq [m]\right\}\sqcup \{V\}.$$
It is ordered by reverse inclusion, viz.\ for every $x,y\in \mathcal{L}(\A)$, we have $x\leq y$ if and only if $x\supseteq y$. The reverse inclusion endows the intersection poset with a rank function, which we denote by $\mathrm{rk}\colon \mathcal{L}(\mathcal{A})\rightarrow \mathbb{Z}_{\geq 0}$. The bottom element of $\mathcal{L}(\A)$ corresponds to $V$, whereas a top element exists if and only if the arrangement is central. We denote the bottom (resp. top) element by $\hat{0}$ (resp. $\hat{1}$). 

Recall the definition of the cone over $\A$, denoted by $c\A$, defined in Section~\ref{Sec:BeyondBoolean}. Let $H_0$ denote the additional hyperplane of $c\A$. For any $x\in \La(\A)$, let $cx\in \La(c\A)$ be the subspace obtained by homogenising the linear affine polynomials corresponding to $x$. Considering the non-empty central $1$--arrangement $\mathcal{B}=(\{0\}, K)$, \cite[Prop.~2.17]{orlik_arrangements_1992} provides a rank-preserving surjective map of posets $$\phi\colon \La(\A\times \mathcal{B}) \to \mathcal{L}(c\A),$$ where $\phi(x\oplus \{0\})=cx\cap H_0$ and $\phi(x\oplus \{K\})=cx$ for every $x\in \La(\A)$.
Whenever $\mathcal{A}$ is central, the map $\phi$ defines a bijection, and thus we have by \cite[Prop.~2.14]{orlik_arrangements_1992} an isomorphism of lattices $$\mathcal{L}(\A)\times \mathcal{L}(\mathcal{B})\simeq \mathcal{L}(c\A).$$ In particular, we have 
\begin{align}
\label{eq:LConeRestr}
    \La((c\A)_{c\hat{1}}) \simeq \La(\A)\times \{K\} &&\text{and}
    &&\mathcal{L}((c\A)^{H_0}) \simeq \mathcal{L}(\A)\times\{0\}.
\end{align}

The order complex of $\La(\A)$ is denoted by $\Delta(\La(\A))$; cf. \cite[Chap.~8.4]{petersen_eulerian_2015}. For any non-empty flag $F=(x_1<\dots <x_l)\in \Delta(\La(\A))$, we set $\min(F):= x_1$ and $\max(F):=x_l$. Moreover, we write $$cF:=(cx_1<\dots <cx_l)\in \Delta(\La(c\A)),$$ and $$cF\cap H_0:=(cx_1\cap H_0<\dots <cx_l\cap H_0)\in \Delta(\La(c\A)).$$ For the empty flag $F=\varnothing$, we set $\min(F):=\max(F):=\varnothing$. The concatenation of ``compatible flags" is explained in the following definition. 
\begin{definition}
    Let $F_1=(x_1<\dots < x_l)$ and $F_2=(y_1<\dots <y_m)$ be flags in $\Delta(\La(\A))$ such that $\max(F_1)\leq \min(F_2)$. The concatenation of $F_1$ and $F_2$ is the flag $$F_1\Vert F_2 = (x_1<\dots <x_l\leq y_1<\dots <y_m)\in \Delta(\La(\A)).$$
\end{definition}
For central arrangements, the isomorphisms of posets (\ref{eq:LConeRestr}) enables one to consider a concatenation of flags in the intersection lattice of the cone over $\A$:
\begin{definition}
\label{Def:ConeFlagConc}
    Let $\A$ be a central hyperplane arrangement and let $F_1, F_2\in \Delta(\La(A))$ such that $\max(F_1)\leq \min(F_2)$. The cone concatenation of $F_1$ and $F_2$ is 
    $$F_1\Vert_{c\A} F_2 = cF_1 \Vert (cF_2\cap H_0) \in \Delta(\La(c\A)).$$
\end{definition}

We set $\tL(\A):= \La(\A)\backslash\{\hat{0}\}$. Going back to the assumptions of Section~\ref{Sec:FlagHilbert},
the analytic zeta function of $\A$ and the flag Hilbert--Poincar\'e series of $\A$ are related by the following result.
\begin{theorem}[{\cite[Thm.~B]{MaglioneVoll/24}}]
\label{ThmfHP}
Let $\A$ be a hyperplane arrangement defined over a number field $K$. For any $x\in \tL(\A)$, let 
\begin{equation}
\label{Eq:gFct}
g_x(\bm{s})=\mathrm{rk}(x)+\sum_{\hat{0}\neq y\leq x}s_y.
\end{equation}
If $\O $ is a cDVR with an $\mathcal{O}_K$-module structure with finite residue field of cardinality $q$ such that $\A$ has good reduction over $\mathbb{F}_q$, then $$\zeta_{\A(\O )}(\bm{s})=\mathsf{fHP}_{\A}\left(-q^{-1}, \left(q^{-g_x(\bm{s})}\right)_{x\in \tL(\A)}\right).$$
\end{theorem}

\section{Ask zeta functions associated with central hyperplane arrangements}
\label{Sec:AskHyp}
Let $R$ be an integral domain and set $K:=\mathrm{Frac}(R)$. Let $\A$ be a central hyperplane arrangement of dimension $d(\A):= \dim(\A)$ defined by a fixed collection of polynomials $$\{L_H(X_1, \dots, X_{d(\A)})\in R[\bm{X}]: H\in \A\}.$$ In this section, given a tuple of nonnegative integers $\bm{\mu}=(\mu_x)_{x\in\La(\A)}$, we construct a module representation, denoted by $\eta_{\A}^{\bm{\mu}}$, which encodes the various incidence relations within $\La(\A)$. In the event that $R=\O$ is a cDVR and an $\mathcal{O}_K$-algebra, our construction provides us with the tools required to prove Theorem~\ref{ThmMainAskIgusa}.

\subsection{Incidence representations} 
Imitating the construction of incidence representations associated with hypergraphs from \cite[Sec.~3.2]{RossmannVoll/24}, we introduce (intersection weighted) \emph{incidence representations} of central hyperplane arrangements. We may assume, without loss of generality, that the linear forms associated with the hyperplanes in $\A$ have coefficients in $R$, viz.\ $L_H\in R[\bm{X}]$ for every $H\in \A$.
\begin{definition}
\label{DefIntWeighIncRep}
    Let $E$ be a finite set and let $\bm{\mu}\colon E\rightarrow \mathcal{L}(\A)$ be an arbitrary set function. Considering the incidence set $${I}(\A, \bm{\mu}):=\{(H, e)\in \A\times E: H\in \A_{\bm{\mu}(e)} \}, $$
  the (intersection-weighted) incidence representation of $\A$ is defined by $$ \eta_{\A}^{\bm{\mu}}\colon \begin{array}{rcl} R{I}(\A, \bm{\mu})&\to &\mathrm{Hom}(R^{d(\mathcal{A})}, R E) \\ (H, e)&\mapsto &(\bm{z}\mapsto L_H(\bm{z})\cdot e) \end{array} .$$
\end{definition}
\begin{remark}
\label{Rem:MatrixLinFormsA}
    Let $\bm{Z}=(Z_{(H, e)})_{(H,e)\in {I}(\A, \bm{\mu})}$ be a set of algebraically independent variables over $R$, and let $\{e_1, \dots, e_{d(\A)}\}$ denote the chosen basis of $R^{d(\A)}$. We denote the matrix of linear forms associated with $\eta_{\A}^{\bm\mu}$ by $$A_{\A}^{\bm\mu}(\bm{Z})\in M_{d(\A)\times \lvert E\rvert }(R[\bm{Z}]);$$ cf. equation~(\ref{eq:MatrixLinForms}). In particular, for any $(i,e)\in [d(\A)]\times E$, we have 
    $$\left(A_{\A}^{\bm\mu}(\bm{Z})\right)_{(i,e)}= \sum_{H\in \mathcal{A}_{\bm\mu(e)}} L_{H}(e_i) Z_{(H, e)}.$$
\end{remark}
The map $\bm\mu\colon E \rightarrow \La(\A)$ in Definition~\ref{DefIntWeighIncRep} is entirely determined by $\mathcal{L}(\mathcal{A})$ and, for each element $x\in \La(\mathcal{A})$, an ``intersection multiplicity" $\mu_x\in \mathbb{Z}_{\geq 0}$ that counts the size of the preimage of $x$ under the map $\bm\mu$:
\begin{equation*}
\mu_x:=\#\{e\in E: \mu(e)=x\}.
\end{equation*}
We note that such intersection multiplicities are analogous to the hyperedge multiplicities associated with hypergraphs; cf. \cite[Def.~3.1]{RossmannVoll/24}. When confusion is unlikely, we often denote a set map $E\rightarrow \La(\A)$, and its corresponding intersection multiplicities, by the same symbol $\bm{\mu}$.

\begin{example}
\label{Ex:BooleanRepr}
    The Boolean hyperplane arrangement of rank $n$, denoted by $\mathrm{A}_1^n$, is comprised of all coordinate hyperplanes in $K^n$. In view of \cite[Ex.~2.8]{orlik_arrangements_1992}, we identify $\La(\mathrm{A}_1^n)$ with the Boolean algebra $B_n$ of all subsets of $[n]$. Given a finite set $E$ and a set function $\lvert \cdot \rvert\colon E\rightarrow B_n$, for any $1\leq i\leq n$ and any $e\in E$, we have  
    $$\left(A_{\mathrm{A}_1^n}^{\lvert\cdot \rvert}(\bm{Z})\right)_{(i, e)}=\begin{cases}
    Z_{(i,e)} &\text{if } i\in \lvert e\rvert, \\ 
    0 &\text{otherwise.}
\end{cases}$$
In particular, in accordance with the notation in \cite[Sec.~3]{RossmannVoll/24}, it follows by construction that the matrix of linear forms $A_{\mathrm{A}_1^n}^{\lvert\cdot \rvert}(\bm{Z})$ coincides with $M_{\mathsf{H}}(\bm{Z})$, where $\mathsf{H} =([n], E, \lvert\cdot \rvert)$ is the hypergraph on $n$ vertices with hyperedge set $E$, and where $\lvert \cdot \rvert\colon E\rightarrow B_n$ plays the role of the support function. 
\end{example}

\begin{example}
    Recall the matrix of linear forms $A_{\operatorname{A}_3}^{\bm{\mu}}(\bm{Z})$ constructed in Example~\ref{ExampleBraidMatrix}.
    The matrix of linear forms associated with its Knuth dual is denoted by $C_{\eta_{\mathcal{A}}^{\bm\mu}}(\bm{X})$, which is of size $14\times 5$; cf. equation~(\ref{KnuthDual}). The rows are indexed by the incidence relations corresponding to the non-dashed edges in Figure \ref{Fig:IntLatticeBraid}, and the columns are indexed by the set $\{123, 124, 12|34, 134, 234\}$. Any entry is either zero, or the linear form determined by the corresponding hyperplane:
    $$ C_{\eta_{\mathcal{A}}^{\bm\mu}}(\bm{X})=\begin{pmatrix}
         X_1-X_2 & 0 & 0 & 0 & 0\\
         X_1-X_3 & 0 & 0 & 0 & 0\\
         X_2-X_3 & 0 & 0 & 0 & 0 \\ 
         0 & X_1-X_2 & 0 & 0 & 0 \\
         0 & X_1-X_4 & 0 & 0 & 0 \\
         0 & X_2-X_4 & 0 & 0 & 0 \\
         0 & 0 & X_1-X_2 & 0 & 0 \\
         0 & 0 & X_3-X_4 & 0 & 0 \\
         0 & 0 & 0 & X_1-X_3 & 0 \\
         0 & 0 & 0 & X_1-X_4 & 0 \\
         0 & 0 & 0 & X_3-X_4 & 0 \\
         0 & 0 & 0 & 0 & X_2-X_3 \\
         0 & 0 & 0 & 0 & X_2-X_4 \\
         0 & 0 & 0 & 0 & X_3-X_4 
    \end{pmatrix}.
    $$
\end{example}

\subsection{Ask zeta functions of incidence representations}
Recall that in the setting of Theorem~\ref{ThmMainAskIgusa}, we have an arrangement $\A$ defined over a cDVR and $\mathcal{O}_K$-algebra $\O$, and $\bm{\mu}=(\mu_x)_{x\in \La(\A)}$ denotes a vector of nonnegative integers.
\begin{proof}[Proof of Theorem~\ref{ThmMainAskIgusa}]
    Let $C_{\eta_{\mathcal{A}}^{\bm\mu}}(\bm{X})$ be the matrix of linear forms associated with the Knuth dual of $\eta_{\A}^{\bm\mu}$; cf. equation~(\ref{KnuthDual}). In view of equation (\ref{Eq:AskCoker}), we have
    \begin{equation}
    \label{EquationHypPadicIntegral}
    (1-q^{-1})\ZA=\int_{\O^{d(\mathcal{A})+1}} \lvert y\rvert^{s-1+m-d(\A)}\left\lvert \mathrm{Coker}\left(C_{\eta_\A^{\bm\mu}}(\bm{X})\right)\otimes_{\O[\bm{X}]} \left(\frac{\O}{y}\right)_{\bm{z}}\right\rvert \d\mu(\bm{z}, y).
    \end{equation}
    Consider the finite set 
    $$E_{\bm\mu}:=\left\{(x, i)\ |\ x\in \La(\mathcal{A}), i\in \{1, \dots, \mu_x\} \right\}$$ endowed with the projection map $\pi_1\colon E_{\bm\mu}\rightarrow \La(\mathcal{A})$. In light of \cite[Lem.~2.3]{RossmannVoll/24}, we have $$\mathrm{Coker}\left(C_{\eta_\A^{\bm\mu}}(\bm{X})\right)=\frac{\O[\bm{X}] E_{\bm\mu}}{\langle L_H(\bm{X}) e : e\in E_{\bm\mu}, H\in \A_{\pi_1(e)}\rangle},$$ which is isomorphic as an $\O[\bm{X}]$-module to 
    \begin{equation*}
    \bigoplus_{e\in E_{\mu}} \frac{\O[\bm{X}]}{\langle L_H(\bm{X}): H\in \A_{\pi_1(e)}\rangle} \simeq \bigoplus_{x\in \La(\A)} \left(\frac{\O[\bm{X}]}{\langle L_H(\bm{X}): H\in \A_x\rangle} \right)^{\mu_x}.
    \end{equation*}
    For any ring homomorphism $\lambda\colon R\rightarrow S$ and any ideal $I\leq R$, we have an isomorphism of $R$-modules $R/I\otimes_R S\simeq S/\lambda(I)S$; cf. \cite[Chap.~2]{AtiyahMacdonald/69}. In particular, the module appearing in the integrand of the right-hand side of equation~(\ref{EquationHypPadicIntegral}) is
    \begin{equation}
    \label{IntModuleA}
    \bigoplus_{x\in \La(\A)} \left(\frac{\O[\bm{X}]}{\langle L_H(\bm{X}): H\in \A_x\rangle} \right)^{\mu_x} \otimes_{\O [\bm{X}]}\left(\frac{\O }{y}\right)_{\bm{z}} \simeq \bigoplus_{x\in \La(\mathcal{A})} \left(\frac{\O }{\langle y, L_H(\bm{z}) \ \vert \  H\in \A_x\rangle}\right)^{\mu_x}.
    \end{equation}
    Combining equations (\ref{EquationHypPadicIntegral}) and (\ref{IntModuleA}) therefore yields
    \begin{align*}
     (1-q^{-1})&\ZA \\ =&\int_{\O^{\d(\mathcal{A})}\times\O } \lvert y\rvert^{s-1+m-d(\A)} \prod_{x\in \La(\A)} \lVert \left\{y, L_H(\bm{z}): H\in \mathcal{A}_x\right\} \rVert^{-\mu_x}\ \d\mu(\bm{z}, y) \nonumber \\  = &\int_{\O^{d(\mathcal{A})}\times\O } \prod_{\underline{x}\in \tL(c\A)} \lVert (c\A)_{\underline{x}}\rVert^{f_{d(\A), \underline{x}}(s,m,\bm{\mu})}\ \d\mu(\bm{z}, y) \nonumber \\ =&\zeta_{c\A(\O)}\left( (f_{d(\A), \underline{x}}(s, m, \bm{\mu}))_{\underline{x}\in \tL(c\A)}\right), \nonumber
    \end{align*}
    which concludes the proof.
    \end{proof}

    \begin{remark}
    \label{RemarkA1andHypergraph}
    Specialising Theorem~\ref{ThmMainAskIgusa} to the Boolean arrangement $\mathrm{A}_1^n$ recovers \cite[Prop.~4.10]{MaglioneVoll/24}.
    Specifically, for any subset $J\subseteq [n+1]$, consider the multivariate polynomial
\begin{equation*}
f_{n, J}(X, Y, (Z_I)_{I\subseteq [n]}):=\begin{cases}
    X+Y-Z_{\varnothing}-n-1 &\text{if $\{n+1\}=J, $} \\
    -Z_{J\backslash \{n+1\}} &\text{if $\{n+1\}\subsetneq J$,} \\
    0 &\text{otherwise.}
\end{cases}
\end{equation*}
In view of Example~\ref{Ex:BooleanRepr}, for any hypergraph $\mathsf{H}$ on $n$ vertices with hyperedge multiplicities $\bm{\mu}=(\mu_I)_{I\subseteq [n]}$, the matrices of linear forms $M_{\mathsf{H}}(\bm{Z})$ and $A_{\mathrm{A}_1^n}^{\bm{\mu}}(\bm{Z})$ coincide. 
Let $\O$ be a cDVR with residue field of size $q$. Upon identifying $c\mathrm{A}_1^n$ with $\mathrm{A}_1^{n+1}$, applying Theorem~\ref{ThmMainAskIgusa} yields
\begin{equation*}
\zeta_{M_\mathsf{H}/\O }^{\mathrm{ask}}(s)=(1-q^{-1})^{-1}\zeta_{\mathrm{A}_1^{n+1}(\O )}((f_{n, J}(s, m, (\mu_I)_{I\subseteq [n]})_{\varnothing \neq J\subseteq [n+1]}),
\end{equation*}
where $m=\sum_{I\subseteq [n]} \mu_I$.
    \end{remark}

\section{Truncated flag Hilbert--Poincar\'e series}
\label{SectionfHPCone}
Our main focus in this section is to prove Theorem~\ref{Thm:mfHPandfHP} and Corollary~\ref{CorollarymfHP}. Along the way, we work out the example of the Boolean arrangement. We assume throughout that $\A$ is a central hyperplane arrangement defined over a number field $K$. We let $d(\A)$ denote the dimension of $\A$, and $r$ denote the rank function on the intersection lattice of $\A$.

\subsection{Rank specialisations} Recall that the truncated flag Hilbert--Poincar\'e series of $\A$ is the rational function $$\mathsf{mfHP}_{\mathcal{A}}(X, (T_x)_{x\in \mathcal{L}(\A)}):=\sum_{F\in \Delta(\mathcal{L}(\mathcal{A}))}\underline{\pi}_F(X) \prod_{x\in F}\frac{T_x}{1-T_x}\in \mathbb{Q}(\bm{T})[X],$$
where $\underline{\pi}_F(X)$ denotes the truncated flag Poincar\'e polynomial; cf. Definition~\ref{Def:TruncfHP}. 
\begin{example}
\label{mfHPBoolean}
    Consider the Boolean arrangement $\operatorname{A}_1^n$. In view of \cite[Prop.~4.5]{MaglioneVoll/24}, we have  $$\mathsf{fHP}_{\operatorname{A}_1^n}(X, \bm{T})=(1+X)^n\sum_{F\in \widetilde{\operatorname{WO}}_n} \prod_{J\in F}\frac{T_J}{1-T_J},$$
    where $\widetilde{\operatorname{WO}}_n:=\Delta(2^{[n]}\backslash \{\varnothing\})$ denotes the poset of weak orders on $n$ elements, excluding the empty set. Since any subarrangement and any restriction of $\mathrm{A}_1^n$ is also Boolean, the truncated flag Hilbert--Poincar\'e series of $\operatorname{A}_1^n$ is
    $$\mathsf{mfHP}_{\operatorname{A}_1^n}(X, \bm{T})=\sum_{F\in \widehat{\operatorname{WO}}_n} (1+X)^{n-\lvert \min(F)\rvert} \prod_{J\in F} \frac{T_J}{1-T_J},$$
    where $\widehat{WO}_n:=\Delta(2^{[n]}).$
\end{example}

\begin{proof}[Proof of Theorem~\ref{Thm:mfHPandfHP}]
    Let $F^{\prime}, F\in \Delta(\La(\A))$ be flags such that $\max(F^{\prime})\leq \min(F)$. In view of Definition~\ref{Def:ConeFlagConc}, we consider the concatenation of flags
    $$F^{\prime}\Vert_{c\A} F := cF^{\prime} \Vert (cF\cap H_0) \in \Delta(\La(c\A)).$$
    In view of the isomorphisms of posets (\ref{eq:LConeRestr}), the right-hand side of equation~(\ref{Eq:substitutionsconefHP}) can be rewritten as
\begin{equation}
\label{AskAMonster}
\sum_{F\in \Delta(\mathcal{L}(\A))}\left(\sum_{\substack{F^{\prime}\in \Delta(\tL(\mathcal{A}))\\ \max(F^{\prime})\leq \min(F)}} \pi_{F^{\prime}\Vert_{c\A} F}(X)\prod_{y\in F^{\prime}}\frac{(-X)^{\operatorname{rk}(y)}}{1-(-X)^{\operatorname{rk}(y)}}\right)\prod_{x\in F}\frac{T_x}{1-T_x}.
\end{equation}
Writing $F=(x_1< \dots <x_{l_1})$ and $F^{\prime}=(y_1 < \dots < y_{l_2})$ in (\ref{AskAMonster}), we have 
\begin{align*}
\pi_{F^{\prime}\Vert_{c\A} F}(X) =& \left(\prod_{i=0}^{l_2-1}\pi_{(c\A)^{cy_{i}}_{cy_{i+1}}}(X) \right)\cdot \pi_{(c\mathcal{A})^{cy_{l_2}}_{cx_1\cap H_0}}(X)\cdot \left( \prod_{j=1}^{l_1}\pi_{(c\A)^{cx_j\cap H_0}_{cx_{j+1}\cap H_0}}(X)\right) \nonumber \\ =& \pi_{(c\mathcal{A})^{cy_{l_2}}_{cx_1\cap H_0}}(X)\cdot\left(\prod_{i=0}^{l_2-1}\pi_{\A^{y_{i}}_{y_{i+1}}}(X)\right) \cdot \left(\prod_{j=1}^{l_1}\pi_{\A^{x_j}_{x_{j+1}}}(X)\right).
\end{align*}
Since $\La((c\A)^{cy_{l_2}}_{cx_1\cap H_0}) \simeq \La(c(\A_{x_1}^{y_{l_2}}))$, applying \cite[Prop.~2.51]{orlik_arrangements_1992} yields $$\pi_{(c\mathcal{A})^{cy_{l_2}}_{cx_1\cap H_0}}(X)= (1+X) \pi_{\A_{x_1}^{y_{l_2}}}(X),$$
and thus 
$$\pi_{F^{\prime}\Vert_{c\A} F}(X)=(1+X) \pi_{\A_{\min(F)}^{\max(F^{\prime})}}(X)\frac{\pi_{F^{\prime}}(X)}{\pi_{\mathcal{A}^{\max(F^{\prime})}}(X)}\frac{\pi_{F}(X)}{\pi_{\mathcal{A}_{\min(F)}}(X)}.$$
Plugging the latter into (\ref{AskAMonster}), we get $$(1+X)\sum_{F\in \Delta(\mathcal{L}(\A))} \frac{\pi_{F}(X)}{\pi_{\A_{\min(F)}}(X)} \mathsf{fHP}_{\A_{\min(F)}}\left(X, ((-X)^{\operatorname{rk}(y)})_{y\in \tL(\mathcal{A}_{\min(F)})}\right) \prod_{x\in F} \frac{T_x}{1-T_x}.$$
To deal with the $\mathsf{fHP}_{\A_{\min(F)}}$ factor appearing in the above summand, we make use of the fact that $\A$ is defined over a field of characteristic zero. In view of the discussion carried out in Section~\ref{Sec:FlagHilbert}, $\A$ is representable over a number field $K$ with ring of integers $\mathcal{O}_K$. Suppose that $\A$ has good reduction over the finite residue field $\mathbb{F}_q$ of $\mathfrak{o}$, where $\O$ is the completion of $\mathcal{O}_K$ at some non-zero prime ideal $\mathfrak{p}\subset \mathcal{O}_K$. We have by Theorem~\ref{ThmfHP} that
$$\mathsf{fHP}_{\A_{\min(F)}}(-q^{-1}, (q^{-\operatorname{rk}(y)})_{y\in \tL(\mathcal{A}_{\min(F)})})= \zeta_{\mathcal{A}_{\min(F)}(\mathfrak{o})}(\bm{0})=\int_{\O^{d(\mathcal{A}_{\min(F)})}} 1 \ \d\mu=1.$$
This concludes the proof, since $\A$ admits good reduction over infinitely many $\mathbb{F}_q$; cf.~\cite[Prop.~5.13]{Stanley/07}.
\end{proof}

\begin{remark}
    The statement of Theorem~\ref{Thm:mfHPandfHP} does not hold for non-central arrangements, since the second isomorphism in (\ref{eq:LConeRestr}) is not applicable in the affine case. Instead, it follows from the proof of Theorem~\ref{Thm:mfHPandfHP} that
    $$(1+X)\mathsf{mfHP}_{(c\mathcal{A})^{H_0}}(X, (T_x)_{x\in \mathcal{L}((c\A)^{H_0})})=\mathsf{fHP}_{c\A}\left(X, ((-X)^{\operatorname{rk}(y)})_{y\not\geq H_0}, (T_x)_{x\geq H_0}\right).$$
\end{remark}

\begin{example}
\label{Ex:HypergraphmfHP}
    Let $\mathsf{H}$ be a hypergraph on $n$ vertices with hyperedge multiplicities $\bm\mu$. In light of Remark~\ref{RemarkA1andHypergraph}, we can make use of the truncated flag Hilbert--Poincar\'e series of the Boolean arrangement $\operatorname{A}_1^n$ as a means to derive the combinatorial formula in terms of weak orders for the rational function $W_{\mathsf{H}}(X, T)$ in \cite[Thm.~C]{RossmannVoll/24}. 
    Concretely, since Boolean arrangements are regular (see \cite[Chap.~6.6]{oxley_matroid_2011}), \emph{for any} cDVR $\O$ with residue field $\mathbb{F}_q$, combining Remark~\ref{RemarkA1andHypergraph} and Theorem~\ref{ThmfHP} yields
    $$\zeta_{M_{\mathsf{H}}/\O}^{\operatorname{ask}}(s)=(1-q^{-1})^{-1}\mathsf{fHP}_{c\mathrm{A}_1^{n}}(-q^{-1}, (q^{-g_J((f_{n, I}(s,m,\bm{\mu}))_{\varnothing\subsetneq I\subseteq [n+1]})})_{\varnothing\subsetneq J\subseteq [n+1]}).$$
    Applying Theorem~\ref{Thm:mfHPandfHP} to the right-hand side gives
$$\zeta_{M_{\mathsf{H}}/\O}^{\operatorname{ask}}(s)=\mathsf{mfHP}_{\operatorname{A}_1^{n}}(-q^{-1}, (q^{-s}q^{-\lvert J\rvert+n-m+\sum_{I\subseteq J}\mu_I})_{J\subseteq [n]}).$$
    In light of Example~\ref{mfHPBoolean}, we obtain
    \begin{align}
    \label{FormulaBooleanHyperplane}
     \zeta_{M_{\mathsf{H}}/\O}^{\operatorname{ask}}(s)=\sum_{F\in \widehat{\operatorname{WO}}_n} (1-q^{-1})^{n-\lvert\min(F)\rvert}\prod_{J\in F} \frac{q^{-\lvert J\rvert +n -m+\sum_{I\subseteq J}\mu_I}q^{-s}}{1-q^{-\lvert J\rvert +n -m+\sum_{I\subseteq J}\mu_I}q^{-s}}.
    \end{align}
    Denoting $J^c$ to be the complement of any subset $J\subseteq [n]$, we have $\lvert J^c\rvert = n-\lvert J\rvert$, whence $$-m+\sum_{I\subseteq J^c}\mu_I= -m+\sum_{I\cap J= \varnothing} \mu_I=-\sum_{I\cap J\neq \varnothing }\mu_I.$$ 
    In particular, applying the bijection $\widehat{\operatorname{WO}}_n\rightarrow \widehat{\operatorname{WO}}_n$ that sends any flag $F$ to its complement $F^c$, the right-hand side of equation~(\ref{FormulaBooleanHyperplane}) can be rewritten to recover \cite[Thm.~C]{RossmannVoll/24}: $$\zeta_{M_\mathsf{H}/\O}^{\operatorname{ask}}(s)= W_{\mathsf{H}}(q, q^{-s}),$$ where $$W_{\mathsf{H}}(X, T):=\sum_{F\in \widehat{\operatorname{WO}}_n} (1-X^{-1})^{\lvert\max(F)\rvert} \prod_{J\in F} \frac{X^{\lvert J\rvert-\sum_{I\cap J\neq \varnothing}\mu_I}T}{1-X^{\lvert J\rvert-\sum_{I\cap J\neq \varnothing}\mu_I}T}.$$
\end{example}

\subsection{Good reduction over finite fields}
Making use of the flag Hilbert--Poincaré series and its truncated analogue, we provide in this section a proof of Corollary~\ref{CorollarymfHP}. Recall that $\O$ denotes a cDVR and $\mathcal{O}_K$-algebra with residue field of size $q$. We assume that $\A$ has good reduction over $\mathbb{F}_q$, and let $\bm{\mu}=(\mu_x)_{x\in \La(\A)}$ be a tuple of nonnegative integers.
\begin{proof}[Proof of Corollary~\ref{CorollarymfHP}]
Let $m=\sum_{x\in \La(\A)}\mu_x$. Since $\A$ has good reduction over $\mathbb{F}_q$, combining Theorem~\ref{ThmMainAskIgusa} and Theorem~\ref{ThmfHP} yields 
\begin{equation*}
        \ZA=(1-q^{-1})^{-1} \mathsf{fHP}_{c\A}\left(-q^{-1}, \left(q^{-g_{\underline{y}}((f_{d(\A), \underline{x}}(s,m, \bm{\mu}))_{\underline{x}\in \tL(c\A)})}\right)_{\underline{y}\in \tL(c\A)}\right).
\end{equation*}
As in Section~\ref{Sec:BeyondBoolean}, given $\underline{y}\in \La((c\A)^{H_0})$, we let $y\in \La(\A)$ denote the corresponding element under the isomorphism (\ref{eq:LConeRestr}). In particular, we have $r_{c\A}(\underline{y})=r_{\A}(y)+1$, and thus $$d(\A)-1+r_{c\A}(\underline{y})=d(\A)-r_\A(y).$$ Combining equations (\ref{FunctionCentralA}) and (\ref{Eq:gFct}) then yields
$$g_{\underline{y}}((f_{d(\A), \underline{x}}(s,m, \bm{\mu}))_{\underline{x}\in \tL(c\A)})=\begin{cases}
    s+m-d(\A) + r_\A(y) - \sum_{x\leq y}\mu_x & \text{if } \underline{y}\gneq H_0, \\
   s+m-d(\A)-\mu_{\hat{0}} &\text{if } \underline{y}=H_0, \\ 
   r_{c\A}(\underline{y}) &\text{otherwise.}
\end{cases}
$$
Using the fact that $m-\sum_{x\leq y}\mu_x=\sum_{x\not\leq y}\mu_x$, applying Theorem \ref{Thm:mfHPandfHP} yields the desired result.
\end{proof}
\begin{remark}
\label{Rem:LocalIgusaHyp}
    Given a central arrangement $\A$ with good reduction over $\mathbb{F}_q$, we may more generally consider complex variables $s_x$ instead of intersection multiplicities $\mu_x\in \mathbb{Z}_{\geq 0}$ for every $x\in \mathcal{\La(\A)}$. Combining Theorem~\ref{Thm:mfHPandfHP} and Theorem~\ref{ThmfHP} yields
    \begin{align*}
    &(1-q^{-1})\mathsf{mfHP}_{\A}(-q^{-1}, (q^{-r(y)-1-\sum_{x\leq y}s_x})_{y\in \mathcal{L}(\A)}) \\=& \mathsf{fHP}_{c\A}(-q^{-1}, (q^{-r(y)})_{y\not\geq H_0}, (q^{-r(y)-1-\sum_{x\leq y}s_x})_{\underline{y}\geq H_0}) \nonumber \\
    =& \int_{\O^{d(\A)}\times\O} \prod_{x\in \mathcal{L}(\A)} \lVert \mathcal{A}_x; y\rVert^{s_x} \d\mu(p, y), \nonumber
    \end{align*}
    where we set $\mathcal{A}_{\hat{0}}=\varnothing$.
    Given a complex variable $s$ and intersection multiplicities $\bm{\mu}$, carrying out the substitutions   $$s_x=\begin{cases}
    s-1+m-d(\A)-\mu_{\hat{0}} &\text{if } x=\hat{0}, \\
    -\mu_x & \text{otherwise,}
    \end{cases}$$ 
    where $m=\sum_{x\in \La(\A)}\mu_x$, recovers Corollary~\ref{CorollarymfHP} in view of Theorem~\ref{ThmMainAskIgusa}.
\end{remark}

\section{Applications}
\label{S:AnalyticProperties}
Let $K$ be a number field, and let $\mathcal{V}_K$ be the set of non-Archimedean places of $K$. For any $v\in \mathcal{V}_K$, we set $K_v$ to be the $v$-adic completion of $K$, and let $\mathcal{O}_v$ denote its valuation ring. The maximal ideal of $\mathcal{O}_v$ is denoted by $\mathfrak{p}_v$, and we let $q_v$ denote the size of the (finite) residue field $\mathbb{F}_{q_v}=\mathcal{O}_v/\mathfrak{p}_v$. Given a central hyperplane arrangement $\mathcal{A}$ defined over $K$ and a tuple of nonnegative integers $\bm{\mu}=(\mu_x)_{x\in \La(\A)}$, we consider the matrix of linear forms constructed in Definition~\ref{DefIntWeighIncRep} over $\mathcal{O}_K$: 
$$A_{\A}^{\bm{\mu}}(\bm{Z})\in M_{d(\A)\times m}(\mathcal{O}_K[\bm{Z}]),$$
where $m:=\sum_{x\in \La(\A)}\mu_x$, and $d(\A):= \dim(\A)$. Throughout this section, the rank function on on the intersection lattice of $\A$ is simply denoted by $r\colon \La(\A)\rightarrow \mathbb{Z}_{\geq 0}$. 

\subsection{Local functional equations}
\label{Sec:FunEquations}
In light of Remark~\ref{Rem:LocalIgusaHyp}, we are considering, in the local framework, $p$-adic integrals of the form $$\int_{\mathcal{O}_v^{d(\A)}\times \mathcal{O}_v} \prod_{x\in \La(\A)} \lVert \A_x; y\rVert^{s_x} \ \d\mu.$$
Such local Igusa zeta functions arising from asymptotic algebra tend to satisfy functional equations under the ``inversion of the prime" $q_v\rightarrow q_v^{-1}$; cf.~\cite{Voll/10}, \cite[Rem.~1.7]{Voll/19}. Whenever such $p$-adic integrals are uniform, meaning that the local zeta function is a rational function in both $q_v^{-s_x}$ and $q_v$, the operation of inverting $q_v$ amounts to formally inverting these parameters. 

In our setting, we can make use of Corollary~\ref{Cor:mfHPSelf} to obtain the local functional equation for $Z_{\bm{\mu}_\A/\O }^{\operatorname{ask}}(T)$ under inversion of both the variable $T$ and the prime power $q_v$. Specifically, if $\mathcal{A}$ has good reduction over $\mathbb{F}_{q_v}$, we carry out, in light of Corollary~\ref{CorollarymfHP}, the substitution $$T_{\hat{1}}:=q_v^{-s}q_v^{d(\A)-r(\A)}.$$ In view of Corollary~\ref{Cor:mfHPSelf}, inverting the prime power $q_v$ yields
$$ Z_{\bm{\mu}_\A/\O }^{\operatorname{ask}}(T)\bigg\vert_{(q_v,T)\rightarrow (q_v^{-1},T^{-1})} = (-q_v^{d(\A)}T)Z_{\bm{\mu}_\A/\O }^{\operatorname{ask}}(T).$$
We remark that this local functional equation is not new: it is a particular case of \cite[Thm.~4.18]{RossmannAsk/18}, which establishes local functional equations under the inversion $q_v\rightarrow q_v^{-1}$ for general ask zeta functions.

\subsection{Analytic properties}
In this subsection, we deduce both local and global analytic properties pertaining to the ask zeta functions associated with $\A$.
\subsubsection{Local poles}
We provide a finite set of integers, defined in terms of the combinatorics of $\A$, that contains the set of all the real parts of the poles of the local ask zeta function of $A_\A^{\bm{\mu}}(\bm{Z})$. Most notably, making use of results in \cite[Sec.~4]{AvniKlopschOnnVoll/13}, we show that this set of candidate (local) poles is valid for every $v\in \mathcal{V}_K$, meaning that it does not depend on whether $\A$ admits good reduction over $\mathbb{F}_{q_v}$. Our description of the local poles will later be leveraged to deduce analytic properties of the global ask zeta functions associated with $\A$. 
\begin{proposition}
\label{Prop:LocalPoles}
    For every $v\in \mathcal{V}_K$, the real parts of the poles of $\zeta^{\operatorname{ask}}_{\bm{\mu}_{\A}/\mathcal{O}_v}(s)$ are contained in 
    $$\mathcal{P}_{\A}^{\bm{\mu}}:=\left\{ d(\A)-r(x)-\sum_{y\not\leq x}\mu_y : x\in \mathcal{L}(\A)\right\}\subseteq \left\{d(\A)-r(\A)-m, \dots, d(\A) \right\},$$
    a set of integers of cardinality at most $\min\{\lvert\La(\A)\rvert, r(\A)+m+1\}$.
\end{proposition}
\begin{proof}
    If $\A$ has good reduction over $\mathbb{F}_{q_v}$, writing the truncated flag Hilbert--Poincar\'e series of $\A$ over a common denominator and carrying out the substitutions described in Corollary~\ref{CorollarymfHP} yields $\mathcal{P}_{\A}^{\bm{\mu}}$. 

    Let $v\in \mathcal{V}_K$ such that $\A$ does not admit good reduction over $\mathbb{F}_{q_v}$, and set $\O:=\mathcal{O}_v$ and $q:=q_v$. On account of Theorem~\ref{ThmMainAskIgusa}, we consider the analytic zeta function $$\zeta_{\mathcal{C}(\O)}(\bm{s})=\int_{\O^{d(\mathcal{C})}}\prod_{x\in \tL(\mathcal{C})}\lVert\mathcal{C} _x\rVert^{s_x} \ \d\mu,$$
    where $\mathcal{C}=c\A$. 
    Crucially, this $p$-adic integral is of the form \cite[Eq.~(4.3)]{AvniKlopschOnnVoll/13}, where we have $I=\varnothing$. Let $(Y, h)$ be a principalisation over $K$ of the ideal $\mathcal{I}_{\mathcal{C}}=\prod_{x\in \tL(\mathcal{C})}(\mathcal{C}_x)$, with numerical data $(N_{tx}, \nu_{t})_{(t,x)\in T\times \tL(\A)}$, where $T$ is a finite index set for the set of irreducible components of $h^{-1}(\cup_{H\in \mathcal{C}} H)$; cf.~\cite[Thm.~3.2.1]{Igusa/00} for the case of prinicpal ideals, and \cite[Sec.~2]{Voll/10} for the general case. In view of \cite[Cor.~4.2]{AvniKlopschOnnVoll/13} and \cite[Prop.~4.5]{AvniKlopschOnnVoll/13}, the study of the poles of $\ZA$ is related to the analysis of the following function: $$\Xi_{U, \varnothing}(q, \bm{s})=\sum_{m\in \mathbb{N}^U} q^{-\sum_{u\in U} m_u\left( \nu_u+\sum_{x\in \tL(\mathcal{C})} s_xN_{u x}\right)} = \prod_{u\in U} \frac{q^{-\nu_u-\sum_{x\in \tL(\mathcal{C})} s_x N_{ux}}}{1-q^{-\nu_u-\sum_{x\in \tL(\mathcal{C})} s_x N_{ux}}},$$
    where $U\subseteq T$. Note that the second equality follows from \cite[Ex.~2.1]{Voll/10}. In particular, the (maximal) De Concini-Procesi wonderful model of $\mathcal{C}$ (cf.~\cite{ConciniProcesi/95})  provides a required principalisation of $\mathcal{I}_C$. This model is described combinatorially in terms of a building set $\mathcal{G}$ of $\mathcal{L}(\mathcal{C})$, along with its nested set complex $\mathcal{N}(\La(\mathcal{C}), \mathcal{G})$; see \cite[Sec.~4]{FeichtnerKozlov/04}. Taking the maximal building set $\mathcal{G}=\tL(\mathcal{C})$, it follows by definition that the nested set complex with respect to $\La(\mathcal{C})$ and $\mathcal{G}$ coincides with $\Delta(\tL(\mathcal{C}))$. Concretely, this means the irreducible components of $h^{-1}(\cup_{H\in \mathcal{C}} H)$ are indexed by $\tL(\mathcal{C})$, and they intersect if and only if their indices are linearly ordered in $\tL(\mathcal{C})$. Given $t, x\in \tL(\mathcal{C})$, the numerical data of the maximal wonderful model of $\mathcal{C}$ is given by $\nu_t=r(t)$, and $$N_{tx}=\begin{cases}
        1 &\text{if $x\leq t$}, \\
        0 &\text{otherwise.}
    \end{cases}$$
    Hence, the function $\Xi_{U, \varnothing}(q, \bm{s})$ has denominators of the form $(1-q^{-g_x(\bm{s})})$ for some $x\in \tL(\mathcal{C})$, where $g_x(\bm{s})$ is the function defined in (\ref{Eq:gFct}).
    Carrying out the substitutions described by Theorem~\ref{ThmMainAskIgusa} concludes the proof.
\end{proof}

\subsubsection{Global abscissa of convergence}
Recall that the global ask zeta function associated with $\A$ is given by the Dirichlet series 
\begin{equation*}
\ZAK=\sum_{I\subseteq \mathcal{O}_K} \operatorname{ask}_{\mathcal{O}_K/I}\left(A_{\A}^{ \bm{\mu}}(\bm{Z})\right)\lvert \mathcal{O}_K/I\rvert^{-s},
\end{equation*}
where $I$ ranges over the non-zero ideals of the ring of integers of $K$. In view of \cite[Prop.~3.4]{RossmannAsk/18}, we have the following Euler product decomposition:
\begin{equation}
\label{Eq:EulerProduct}
\ZAK=\prod_{v\in \mathcal{V}_K} \zeta^{\operatorname{ask}}_{\bm{\mu}_{\A}/\mathcal{O}_v}(s).
\end{equation}
We denote the abscissa of convergence of (\ref{Eq:EulerProduct}) by $\alpha(\A, \bm{\mu})$. We first focus on the case of so-called regular arrangements before turning our attention to arbitrary central arrangements.

A central arrangement $\A$ is regular if it is representable over every field; see \cite[Chap.~6.6]{oxley_matroid_2011}. Equivalently, we let $\A$ be a totally unimodular arrangement, meaning that $\A$ is represented by a matrix with entries in $\mathbb{R}$, such that every minor lies in the set $\{-1,0,1\}$; cf.~\cite[Thm.~6.6.3]{oxley_matroid_2011}. The class of regular arrangements includes, in particular, all Boolean and graphical arrangements. It follows by definition that regular arrangements admit good reduction over every finite field $\mathbb{F}_q$, which enables one to make use of Corollary~\ref{CorollarymfHP} to study analytic properties of the global ask zeta function. Adapting the proof of \cite[Thm.~5.27]{RossmannVoll/24} to our setting, we show that (\ref{Eq:EulerProduct}) admits global analytic properties analogous to the ask zeta functions associated with hypergraphs.
\begin{proposition}
\label{Prop:AbsRegular}
Let $\A$ be a regular arrangement defined over a number field $K$. Let $\bm{\mu}=(\mu_x)_{x\in \La(\A)}$ be a tuple of nonnegative integers. The global abscissa of convergence of $\zeta^{\operatorname{ask}}_{A_{\A}^{\bm{\mu}}/\mathcal{O}_K}(s)$ is a positive integer. It satisfies $$ \alpha(\A, \bm{\mu})=\max\left\{d(\A)+1-r(x)-\sum_{y\not\leq x}\mu_y : x\in \La(\A)\right\},$$ and is independent of the number field $K$.
\end{proposition}
\begin{proof}
    Consider the bivariate rational function 
    \begin{equation}
    \label{eq:BivariateWA}
    W_{\A}(X, T):= \mathsf{mfHP}_{\A}\left(-X^{-1}, (X^{d(\A)-r(y)-\sum_{x\not\leq y}\mu_x}T)_{y\in \La(\A)}\right)\in \mathbb{Q}(X, T),
    \end{equation}
    so that by Corollary~\ref{CorollarymfHP}, we have $\zeta^{\operatorname{ask}}_{\bm{\mu}_\A/\mathcal{O}_v}(s)=W_{\A}(q_v, q_v^{-s})$ for every $v\in \mathcal{V}_K$. Since $\A$ is central, equation~(\ref{eq:mfHPcentral}) shows that
    $$
    W_{\A}(X, T)=\frac{1}{1-X^{d(\A)-r(\A)}T}\left(1+\sum_{\substack{F\in \Delta(\La(\A)\backslash\{\hat{1}\}) \\ F\neq \varnothing }} \underline{\pi}_F(-X^{-1})\prod_{x\in F} \frac{X^{A_{x}}T}{1-X^{A_{x}}T}\right),$$
    where $\underline{\pi}_F(X)\in \mathbb{Z}[X]$ with $\underline{\pi}_F(0)=1$, and $A_{x}\in \mathcal{P}_{\A}^{\bm{\mu}}$; cf. Proposition~\ref{Prop:LocalPoles}. Writing $W_{\A}(X, T)$ over a common denominator yields
    $$\frac{\prod_{x\in \La(\A)\backslash\{\hat{1}\}}(1-X^{A_x}T) +\sum_{F} \underline{\pi}_F(-X^{-1}) \prod_{x\in F}X^{A_{x}}T\prod_{y\notin F} (1-X^{A_{y}}T)}{(1-X^{d(\A)-r(\A)}T)\prod_{x\in \La(\A)\backslash\{\hat{1}\}}(1-X^{A_{ij}}T)}.$$
    Let $\zeta_K(s)$ denote the Dedekind zeta function of $K$. The Euler product
    $$\prod_{v\in \mathcal{V}_K} \frac{1}{\prod_{x\in \La(\A)}(1-q_v^{A_{x}}q_v^{-s})}=\zeta_K(s-d(\A)+r(\A))\prod_{x\in \La(\A)\backslash\{\hat{1}\}}\zeta_K(s-A_{x}).$$
    has abscissa of convergence $$\alpha:=\max\left(d(\A)-r(\A)+1, \max\{A_{x}+1 : x\in \tL(\A)\} \right)\leq d(\A)+1.$$ The remainder of the proof is identical, mutatis mutandis, to that of \cite[Thm.~5.27]{RossmannVoll/24}.
\end{proof}

We now consider $\A$ to be a central hyperplane arrangement (possibly non-regular) defined over $K$. In view of \cite[Prop.~5.13]{Stanley/07}, $\A$ has good reduction over $\mathbb{F}_{q_v}$ for all but finitely many places $v\in \mathcal{V}_K$. We let $S_{\A}\subset \mathcal{V}_K$ denote the finite set of places with bad reduction.

\begin{proof}[Proof of Theorem~\ref{Thm:GlobAnalProp}]
    Combining equation~(\ref{Eq:EulerProduct}) and Corollary~\ref{CorollarymfHP}, we have $$\ZAK= \prod_{v\in \mathcal{V}_K\backslash S_\A}W_{\A}(q_v, q_v^{-s}) \prod_{w\in S_\A}\zeta^{\operatorname{ask}}_{\bm{\mu}_\A/\mathcal{O}_w}(s) ,$$
    where $W_{\A}(X, T)\in \mathbb{Q}(X, T)$ is the rational function defined by (\ref{eq:BivariateWA}).
    On account of Proposition~\ref{Prop:AbsRegular}, the product $\prod_{v\in \mathcal{V}_K\backslash S_\A}W_{\A}(q_v, q_v^{-s})$ has abscissa of convergence $\alpha$ with the desired properties. Moreover, for every $w\in S_\A$, the abscissa of convergence of $\zeta^{\operatorname{ask}}_{\bm{\mu}_\A/\mathcal{O}_w}(s)$ is strictly less than $\alpha$ in view of Proposition~\ref{Prop:LocalPoles}. This concludes the proof since $S_\A$ is finite.
\end{proof}

\subsection{Reduced and topological relatives}
We determine the reduced and topological zeta functions associated with $\zeta^{\operatorname{ask}}_{\bm{\mu}_\A/\mathcal{O}_v}(s)$, which we denote by $Z^{\operatorname{red}}_{\bm{\mu}_\A}(T)$, and $\zeta^{\operatorname{top}}_{\bm{\mu}_\A}(s)$ respectively. Informally speaking, these functions are obtained as two different types of limits as $q_v\to 1$; see \cite[Sec.~7]{RossmannComputing/18}, \cite[Sec.~5]{Rossmann/15} and references therein. 

The reduced relatives of all ask zeta functions have been established in \cite[Sec.~4.6]{RossmannAsk/18}, whence  $$Z^{\operatorname{red}}_{\bm{\mu}_\A}(T)=\frac{1}{1-T}.$$ This result can also be deduced from Corollary~\ref{CorollarymfHP}: consider the bivariate rational function $$\mathsf{mfHP}_{\A}(-X^{-1}, (X^{d(\A)-r(y)-\sum_{x\not\leq y}\mu_x}T)_{y\in \La(\A)})\in \mathbb{Q}(X, T)$$ and set $X=1$. The result being $1/(1-T)$ simply reflects the fact that the Poincar\'e polynomial $\pi_{\mathcal{B}}(X)$ of any central, non empty arrangement $\mathcal{B}$ is divisible by $(1+X)$; cf.~\cite[Prop.~2.51]{orlik_arrangements_1992}.

Adapting the proof of \cite[Cor.~1.5]{MaglioneVoll/24} enables one to obtain the topological relatives. Given any flag $F\in \Delta(\La(\A))$, we define $$\underline{\pi}^{\circ}_F(X):= \frac{\underline{\pi}_F(X)}{(1+X)^{\lvert F\rvert}}\in \mathbb{Z}[X],$$
which is a polynomial in view of \cite[Lem.~2.3]{MaglioneVoll/24}. 
\begin{proposition}
    Let $\A$ be a central hyperplane arrangement defined over a field of characteristic zero, and let $\bm{\mu}=(\mu_x)_{x\in \La(\A)}$ be a tuple of nonnegative integers. Then $$\zeta^{\operatorname{top}}_{\bm{\mu}_\A}(s)=\sum_{F\in \Delta(\La(\A))} \underline{\pi}^{\circ}_F(-1) \prod_{x\in F} \frac{1}{s-d(\A)+r(x)+\sum_{y\not\leq x}\mu_y}.$$
\end{proposition}
\begin{proof}
    Let $Q$ be an indeterminate. For any complex variable $\alpha$, by means of the binomial series, we consider the power series
    $$\frac{Q^{\alpha}-1}{Q-1}=\frac{(1+(Q-1))^{\alpha}-1}{Q-1}=\sum_{k=0}^{\infty} \binom{\alpha}{k+1}(Q-1)^k \in \mathbb{Q}(\alpha)\llbracket Q-1\rrbracket,$$
    which has constant coefficient $\alpha$.
    Furthermore, given any $f\in \mathbb{Z}[Q]$, the constant term of $f(-Q^{-1})\in \mathbb{Q}\llbracket Q-1\rrbracket$ is given by $ f(-1)$. Indeed, for any positive integer $n$, we have $(-Q^{-1})^n=(-1)^nQ^{-n}$, and the constant term of $Q^{\alpha}\in \mathbb{Q}(\alpha)\llbracket Q-1\rrbracket$ is simply $1$.
    This justifies the second equality of the following computation:
    \begin{align}
    \label{eq:topcstterm}
    &\mathsf{mfHP}_{\A}\left(-Q^{-1}, (Q^{-s+d(\A)-r(x)-\sum_{y\not\leq x}\mu_y})_{x\in\La(\A)}\right) \nonumber \\= &\sum_{F\in \Delta(\La(\A))}Q^{-\lvert F\rvert}\underline{\pi}_F^{\circ}(-Q^{-1})\prod_{x\in F} \frac{Q-1}{Q^{s-d(\A)+r(x)+\sum_{y\not\leq x}\mu_y}-1} \nonumber\\ =&\sum_{F\in \Delta(\La(\A))} \underline{\pi}_F^{\circ}(-1) \prod_{x\in F}\frac{1}{s-d(\A)+r(x)+\sum_{y\not\leq x}\mu_y} \ + O((Q-1)).
    \end{align}
    By \cite[Def.~5.13]{Rossmann/15}, the topological zeta function is given by the $l$-adic limit $$\zeta^{\operatorname{top}}_{\bm{\mu}_\A}(s)=\lim_{\substack{f\to 0\\ f\in \mathbb{N}}}\zeta^{\operatorname{ask}}_{{\bm{\mu}_\A}/\O^{(f)}}(s),$$ where $\A$ has good reduction over the residue field of $\O$, and $\O^{(f)}$ denotes the unramified extension of $\O$ of degree $f$. Hence, in view of Corollary~\ref{CorollarymfHP}, the topological zeta function is given by the constant term in (\ref{eq:topcstterm}).
\end{proof}

\section*{Acknowledgments}
The author is grateful to Joshua Maglione, Tobias Rossmann, and Christopher Voll for many helpful discussions and valuable suggestions. This research was initiated during the program - Combinatorial Methods in Enumerative Algebra (code: ICTS/CMEA2024/12) at the International Centre for Theoretical Sciences (ICTS).

\bibliographystyle{abbrv} 
\bibliography{references}

@article {MaglioneVoll/24,
    AUTHOR = {Maglione, Joshua and Voll, Christopher},
     TITLE = {Flag {H}ilbert-{P}oincar\'e{} series and {I}gusa zeta
              functions of hyperplane arrangements},
   JOURNAL = {Israel J. Math.},
  FJOURNAL = {Israel Journal of Mathematics},
    VOLUME = {264},
      YEAR = {2024},
    NUMBER = {1},
     PAGES = {177--233},
      ISSN = {0021-2172,1565-8511},
   MRCLASS = {14N20 (11M41 14E18 52C35)},
  MRNUMBER = {4842997},
       DOI = {10.1007/s11856-024-2646-5},
       URL = {https://doi.org/10.1007/s11856-024-2646-5},
}

@article {RossmannVoll/24,
    AUTHOR = {Rossmann, Tobias and Voll, Christopher},
     TITLE = {Groups, graphs, and hypergraphs: average sizes of kernels of
              generic matrices with support constraints},
   JOURNAL = {Mem. Amer. Math. Soc.},
  FJOURNAL = {Memoirs of the American Mathematical Society},
    VOLUME = {294},
      YEAR = {2024},
    NUMBER = {1465},
     PAGES = {v+120},
      ISSN = {0065-9266,1947-6221},
      ISBN = {978-1-4704-6868-2; 978-1-4704-7753-0},
   MRCLASS = {11M41 (05A15 05C50 11S80 14M25 15B33 20D15 20E45)},
  MRNUMBER = {4719079},
       DOI = {10.1090/memo/1465},
       URL = {https://doi.org/10.1090/memo/1465},
}

@book {AtiyahMacdonald/69,
    AUTHOR = {Atiyah, M. F. and Macdonald, I. G.},
     TITLE = {Introduction to commutative algebra},
 PUBLISHER = {Addison-Wesley Publishing Co., Reading, Mass.-London-Don
              Mills, Ont.},
      YEAR = {1969},
     PAGES = {ix+128},
   MRCLASS = {13.00},
  MRNUMBER = {242802},
MRREVIEWER = {Johnny\ A.\ Johnson},
}

@book{Igusa/00,
	series = {{AMS}/{IP} studies in advanced mathematics},
	title = {An {Introduction} to the {Theory} of {Local} {Zeta} {Functions}},
	isbn = {978-0-8218-2907-3},
	url = {https://books.google.de/books?id=G7BJAwAAQBAJ},
	publisher = {American Mathematical Society},
	author = {Igusa, J.},
	year = {2000},
	lccn = {99087031},
}

@article{Stanley/07,
  title={An introduction to hyperplane arrangements},
  author={Stanley, Richard P},
  journal={Geometric combinatorics},
  volume={13},
  pages={389--496},
  year={2007}
}

@book{orlik_arrangements_1992,
	address = {Berlin, Heidelberg},
	series = {Grundlehren der mathematischen {Wissenschaften}},
	title = {Arrangements of {Hyperplanes}},
	volume = {300},
	isbn = {978-3-642-08137-8 978-3-662-02772-1},
	url = {http://link.springer.com/10.1007/978-3-662-02772-1},
	publisher = {Springer},
	author = {Orlik, Peter and Terao, Hiroaki},
	year = {1992},
	doi = {10.1007/978-3-662-02772-1},
}

@article {RossmannAsk/18,
    AUTHOR = {Rossmann, Tobias},
     TITLE = {The average size of the kernel of a matrix and orbits of
              linear groups},
   JOURNAL = {Proc. Lond. Math. Soc. (3)},
  FJOURNAL = {Proceedings of the London Mathematical Society. Third Series},
    VOLUME = {117},
      YEAR = {2018},
    NUMBER = {3},
     PAGES = {574--616},
      ISSN = {0024-6115,1460-244X},
   MRCLASS = {11S80 (05A15 11M41 20D15 20E45)},
  MRNUMBER = {3857694},
MRREVIEWER = {Bertin\ Diarra},
       DOI = {10.1112/plms.12159},
       URL = {https://doi.org/10.1112/plms.12159},
}

@book{apostol_modular_1990,
	address = {New York, NY},
	series = {Graduate {Texts} in {Mathematics}},
	title = {Modular {Functions} and {Dirichlet} {Series} in {Number} {Theory}},
	volume = {41},
	copyright = {http://www.springer.com/tdm},
	isbn = {978-1-4612-6978-6 978-1-4612-0999-7},
	url = {http://link.springer.com/10.1007/978-1-4612-0999-7},
	urldate = {2025-04-29},
	publisher = {Springer},
	author = {Apostol, Tom M.},
	year = {1990},
	doi = {10.1007/978-1-4612-0999-7},
}

@book{oxley_matroid_2011,
	title = {Matroid {Theory}},
	isbn = {978-0-19-856694-6},
	url = {https://doi.org/10.1093/acprof:oso/9780198566946.001.0001},
	publisher = {Oxford University Press},
	author = {Oxley, James},
	month = feb,
	year = {2011},
	doi = {10.1093/acprof:oso/9780198566946.001.0001},
	doi = {10.1093/acprof:oso/9780198566946.001.0001},
}

@book{petersen_eulerian_2015,
	address = {New York, NY},
	series = {Birkhäuser {Advanced} {Texts} {Basler} {Lehrbücher}},
	title = {Eulerian {Numbers}},
	copyright = {https://www.springernature.com/gp/researchers/text-and-data-mining},
	isbn = {978-1-4939-3090-6 978-1-4939-3091-3},
	url = {https://link.springer.com/10.1007/978-1-4939-3091-3},
	language = {en},
	urldate = {2025-10-20},
	publisher = {Springer},
	author = {Petersen, T. Kyle},
	year = {2015},
	doi = {10.1007/978-1-4939-3091-3},
}

@article {Stump/25,
    AUTHOR = {Stump, Christian},
     TITLE = {Chow and augmented {C}how polynomials as evaluations of
              {P}oincar\'e-extended {${\bf ab}$}-indices},
   JOURNAL = {Adv. Math.},
  FJOURNAL = {Advances in Mathematics},
    VOLUME = {482},
      YEAR = {2025},
     PAGES = {Paper No. 110618, 13},
      ISSN = {0001-8708,1090-2082},
   MRCLASS = {06A07 (05B35 52C40)},
  MRNUMBER = {4975905},
MRREVIEWER = {Joseph\ Kung},
       DOI = {10.1016/j.aim.2025.110618},
       URL = {https://doi.org/10.1016/j.aim.2025.110618},
}

@article {Rossmann/15,
    AUTHOR = {Rossmann, Tobias},
     TITLE = {Computing topological zeta functions of groups, algebras, and
              modules, {I}},
   JOURNAL = {Proc. Lond. Math. Soc. (3)},
  FJOURNAL = {Proceedings of the London Mathematical Society. Third Series},
    VOLUME = {110},
      YEAR = {2015},
    NUMBER = {5},
     PAGES = {1099--1134},
      ISSN = {0024-6115,1460-244X},
   MRCLASS = {11M41 (14E18 14G10 17A01 20F69 52B20)},
  MRNUMBER = {3349788},
MRREVIEWER = {Michael\ M.\ Schein},
       DOI = {10.1112/plms/pdv012},
       URL = {https://doi.org/10.1112/plms/pdv012},
}

@article {RossmannAskII/20,
    AUTHOR = {Rossmann, Tobias},
     TITLE = {The average size of the kernel of a matrix and orbits of
              linear groups, {II}: duality},
   JOURNAL = {J. Pure Appl. Algebra},
  FJOURNAL = {Journal of Pure and Applied Algebra},
    VOLUME = {224},
      YEAR = {2020},
    NUMBER = {4},
     PAGES = {106203, 28},
      ISSN = {0022-4049,1873-1376},
   MRCLASS = {20D15 (05A15 11M41 11S80 13E05 15B33 20E45)},
  MRNUMBER = {4021917},
MRREVIEWER = {Andrei\ Jaikin-Zapirain},
       DOI = {10.1016/j.jpaa.2019.106203},
       URL = {https://doi.org/10.1016/j.jpaa.2019.106203},
}

@article {KuhneMaglione/23,
    AUTHOR = {K\"uhne, Lukas and Maglione, Joshua},
     TITLE = {On the geometry of flag {H}ilbert-{P}oincar\'e{} series for
              matroids},
   JOURNAL = {Algebr. Comb.},
  FJOURNAL = {Algebraic Combinatorics},
    VOLUME = {6},
      YEAR = {2023},
    NUMBER = {3},
     PAGES = {623--638},
      ISSN = {2589-5486},
   MRCLASS = {05B35 (52C40)},
  MRNUMBER = {4614155},
MRREVIEWER = {Winfried\ Hochst\"attler},
       DOI = {10.5802/alco.276},
       URL = {https://doi.org/10.5802/alco.276},
}

@article {DorpalenMaglioneStump/25,
    AUTHOR = {Dorpalen-Barry, Galen and Maglione, Joshua and Stump,
              Christian},
     TITLE = {The {P}oincar\'e-extended {$\mathbf{ab}$}-index},
      NOTE = {With an appendix by Ricky Ini Liu},
   JOURNAL = {J. Lond. Math. Soc. (2)},
  FJOURNAL = {Journal of the London Mathematical Society. Second Series},
    VOLUME = {111},
      YEAR = {2025},
    NUMBER = {1},
     PAGES = {Paper No. e70054, 33},
      ISSN = {0024-6107,1469-7750},
   MRCLASS = {06A07 (05E05 52C40)},
  MRNUMBER = {4843584},
MRREVIEWER = {Brigitte\ Servatius},
       DOI = {10.1112/jlms.70054},
       URL = {https://doi.org/10.1112/jlms.70054},
}

@article {AvniKlopschOnnVoll/13,
    AUTHOR = {Avni, Nir and Klopsch, Benjamin and Onn, Uri and Voll,
              Christopher},
     TITLE = {Representation zeta functions of compact {$p$}-adic analytic
              groups and arithmetic groups},
   JOURNAL = {Duke Math. J.},
  FJOURNAL = {Duke Mathematical Journal},
    VOLUME = {162},
      YEAR = {2013},
    NUMBER = {1},
     PAGES = {111--197},
      ISSN = {0012-7094,1547-7398},
   MRCLASS = {22E50 (11M41 20C15 20G25 22E40)},
  MRNUMBER = {3011874},
MRREVIEWER = {A.\ H.\ Dooley},
       DOI = {10.1215/00127094-1959198},
       URL = {https://doi.org/10.1215/00127094-1959198},
}

@article {Voll/10,
    AUTHOR = {Voll, Christopher},
     TITLE = {Functional equations for zeta functions of groups and rings},
   JOURNAL = {Ann. of Math. (2)},
  FJOURNAL = {Annals of Mathematics. Second Series},
    VOLUME = {172},
      YEAR = {2010},
    NUMBER = {2},
     PAGES = {1181--1218},
      ISSN = {0003-486X,1939-8980},
   MRCLASS = {20E07 (11S40 16P90 39B52)},
  MRNUMBER = {2680489},
MRREVIEWER = {Alexander\ Fel\cprime shtyn},
       DOI = {10.4007/annals.2010.172.1185},
       URL = {https://doi.org/10.4007/annals.2010.172.1185},
}

@article {Voll/19,
    AUTHOR = {Voll, Christopher},
     TITLE = {Local functional equations for submodule zeta functions
              associated to nilpotent algebras of endomorphisms},
   JOURNAL = {Int. Math. Res. Not. IMRN},
  FJOURNAL = {International Mathematics Research Notices. IMRN},
      YEAR = {2019},
    NUMBER = {7},
     PAGES = {2137--2176},
      ISSN = {1073-7928,1687-0247},
   MRCLASS = {11M41 (11S40 16W20 20E07)},
  MRNUMBER = {3938319},
MRREVIEWER = {Shin-ya\ Koyama},
       DOI = {10.1093/imrn/rnx186},
       URL = {https://doi.org/10.1093/imrn/rnx186},
}

@article {RossmannComputing/18,
    AUTHOR = {Rossmann, Tobias},
     TITLE = {Computing local zeta functions of groups, algebras, and
              modules},
   JOURNAL = {Trans. Amer. Math. Soc.},
  FJOURNAL = {Transactions of the American Mathematical Society},
    VOLUME = {370},
      YEAR = {2018},
    NUMBER = {7},
     PAGES = {4841--4879},
      ISSN = {0002-9947,1088-6850},
   MRCLASS = {11M41 (20C15 20F18 20F69 20G30)},
  MRNUMBER = {3812098},
MRREVIEWER = {Asif\ Zaman},
       DOI = {10.1090/tran/7361},
       URL = {https://doi.org/10.1090/tran/7361},
}

@article {ConciniProcesi/95,
    AUTHOR = {De Concini, C. and Procesi, C.},
     TITLE = {Wonderful models of subspace arrangements},
   JOURNAL = {Selecta Math. (N.S.)},
  FJOURNAL = {Selecta Mathematica. New Series},
    VOLUME = {1},
      YEAR = {1995},
    NUMBER = {3},
     PAGES = {459--494},
      ISSN = {1022-1824,1420-9020},
   MRCLASS = {14D99 (32G13 52B30)},
  MRNUMBER = {1366622},
MRREVIEWER = {V.\ Leksin},
       DOI = {10.1007/BF01589496},
       URL = {https://doi.org/10.1007/BF01589496},
}

@article {FeichtnerKozlov/04,
    AUTHOR = {Feichtner, Eva-Maria and Kozlov, Dmitry N.},
     TITLE = {Incidence combinatorics of resolutions},
   JOURNAL = {Selecta Math. (N.S.)},
  FJOURNAL = {Selecta Mathematica. New Series},
    VOLUME = {10},
      YEAR = {2004},
    NUMBER = {1},
     PAGES = {37--60},
      ISSN = {1022-1824,1420-9020},
   MRCLASS = {06A07 (05E99 14E15 14M25 52C35)},
  MRNUMBER = {2061222},
MRREVIEWER = {Henry\ K.\ Schenck},
       DOI = {10.1007/s00029-004-0298-1},
       URL = {https://doi.org/10.1007/s00029-004-0298-1},
}

\end{document}